\documentclass[oneside,english,reqno]{amsart}
\usepackage{algorithm}
\usepackage{algorithmic}
\usepackage{multirow}
\usepackage{subfigure}
\usepackage{enumerate}

\usepackage[T1]{fontenc}
\usepackage[latin9]{inputenc}
\usepackage{babel}
\usepackage{amsmath}
\usepackage{amsthm}
\usepackage{amssymb}
\usepackage{graphicx}
\usepackage[unicode=true,pdfusetitle,
bookmarks=true,bookmarksnumbered=false,bookmarksopen=false,
breaklinks=false,pdfborder={0 0 1},backref=false,
hidelinks]
{hyperref}

\makeatletter
\numberwithin{equation}{section}
\numberwithin{figure}{section}

\makeatother

\newtheorem{theorem}{Theorem}[section]

\newtheorem{corollary}[theorem]{Corollary}

\newtheorem{definition}[theorem]{Definition}
\newtheorem{example}[theorem]{Example}

\newtheorem{lemma}{Lemma}[section]

\newtheorem{proposition}[theorem]{Proposition}
\newtheorem{remark}[theorem]{Remark}

\newcommand{\funf}{f}
\newcommand{\fung}{g}
\newcommand{\gph}{\text{gph}\,}
\newcommand{\hilbertH}{ \mathcal{H}}
\newcommand{\hilbertG}{ \mathcal{G}}
\newcommand{\dist}{\text{dist}}
\newcommand{\rank}{\text{rank}\,}
\newcommand{\multifC}{{\mathbb C}}
\newcommand{\multifR}{{\mathcal  R}}
\newcommand{\setD}{\mathcal{D}}
\usepackage{xcolor}
\begin{document}

\title[On Lipschitz continuity of projections]{On Lipschitz continuity of projections onto polyhedral moving sets}


\author{
	Ewa M. Bednarczuk$^1$
}

\author{
	Krzysztof E. Rutkowski$^2$ 
}
\thanks{$^1$ Systems Research Institute of the Polish Academy of Sciences, Warsaw University of Technology, Faculty of Mathematics and Information Science}

\thanks{$^2$ Cardinal Stefan Wyszy\'nski University, 
	Faculty of Mathematics and Natural Sciences. School of Exact Sciences,
	Warsaw, Poland, 
	Systems Research Institute of the Polish Academy of Sciences}

\begin{abstract}
In Hilbert space setting we prove local lipchitzness of projections onto parametric polyhedral sets represented as solutions to systems of inequalities and equations with parameters appearing both in  left-hand-sides and right-hand-sides of the constraints. In deriving main results we assume that data are locally Lipschitz functions of parameter and  the relaxed constant rank constraint qualification condition is satisfied.
\end{abstract}
\maketitle

\section{Introduction}

Continuity of metric projections of a given $\bar{v}$ onto moving subsets have  already been investigated in a number of  instances. In the framework of Hilbert spaces, the projection $P_{C}(\bar{v})$ of $\bar{v}$ onto closed convex sets $C,C^\prime$, i.e.,  solutions to optimization problems
\begin{equation*}
	\tag{Proj} \text{minimize}\ \|z-\bar{v}\|\ \ \text{subject to \ } z\in C,
\end{equation*}
are unique and H\"older continuous with the exponent 1/2 in the sense that  there exists a constant $\ell_{H}>0$ with
$$
\|P_{C}(\bar{v})-P_{C'}(\bar{v})\|\le \ell_{H}[d_\rho(C,C')]^{1/2},
$$
where $d_\rho(\cdot,\cdot)$ denotes the bounded Hausdorff distance  (see  \cite{quantitive_stablilit_of_variational_systems_II_Attouch_Wets} and also \cite[Example 1.2]{perurbations_approximations_sensitivity_Dontchev}).

In the case where the sets, on which we project a given $\bar{v}$, are solution sets to systems of equations and inequalities,
the problem $Proj$ is a special case of a general parametric problem
\begin{equation}\label{problem:par}
	\tag{Par}\begin{array}[t]{l}\text{minimize  }  \varphi_{0}(p,x)
		\text{   subject to  } \\ \varphi_{i}(p,x)=0\ \ i\in I_{1},\ \
		\varphi_{i}(p,x)\le 0\ \ i\in I_{2},
	\end{array}
\end{equation}
where $x\in \hilbertH$,  $p\in \setD\subset \hilbertG$, $\hilbertH$ - Hilbert space, $\hilbertG$ - metric space, $\varphi_i:\ \setD\times \hilbertH \rightarrow \mathbb{R}$, $i\in \{0\}\cup I_1\cup I_2$. There exist numerous results concerning continuity of solutions to  problem \eqref{problem:par} in finite dimensional settings, see e.g., \cite{Bonnans_Shapiro,Holder_behaviour_of_optimal_soultions_Minchenko_Sakolchik,MR3101098,directional_derivatives_of_the_solutions_o_a_paarametric_nonlinear_Ralph_Dempe} and the references therein. In a recent paper by Mordukhovich, Nghia \cite{full_lipschitzian_and_holderian_stability_Mordukhovich}, in the finite-dimensional setting, the h\"olderness and the lipschitzness 
of the local minimizers to problem $Par$ with $I_{1}=\emptyset$ are investigated for 
$C^{2}$ functions $\varphi_{i}$, $i\in  I_{2}$,  under Mangasarian-Fromowitz (MFCQ)
and constant rank (CRCQ) constraints qualifications.

Let $\hilbertH$ be a Hilbert space and let $\setD\subset \hilbertG$ be a nonempty set of a normed space $\hilbertG$, $p\in \setD$ and $v\in \hilbertH$. We consider the norm topology induced on $\setD$ by the topology of space $\hilbertG$, i.e., $U$ is an open set in $\setD$ if $U=\setD\cap U^\prime$, where $U^\prime$ is open in $\hilbertG$ (see e.g. \cite{MR1039321}).  

We consider the following parametric optimization problem
\begin{align}\label{M(v,p)}\tag{M($v$,$p$)}
	\begin{aligned}
		&\min_{x\in \hilbertH}\, \frac{1}{2}\|x-v\|^2,\\
		&\text{subject to}\ x\in C(p):=\left\{ x\in \hilbertH\ \bigg|\ 
		\begin{array}{ll}
			\langle x \ |\ g_i(p)\rangle = f_i(p), &i \in I_1,\\
			\langle x \ |\ g_i(p)\rangle \leq f_i(p), &i \in I_2
		\end{array}
		\right\},
	\end{aligned}
\end{align}
where $f_i:\ \setD\rightarrow \mathbb{R}$, $g_i:\ \setD\rightarrow \hilbertH$, $i\in  I_1\cup I_2\neq \emptyset$, $I_1=\emptyset\vee \{1,\dots,q\}$, $I_2=\emptyset \vee \{q+1,\dots,m\}$ are locally Lipschitz on $\setD$. 

When $C(p)\neq \emptyset$ for $p\in \setD$, problem $\eqref{M(v,p)}$ is uniquely solvable for any $v\in \hilbertH$ and the solution $P(v,p)$ to 
problem \eqref{M(v,p)}	is the projection of $v$ onto $C(p)$ i.e.
\begin{equation}\label{projection}
	P(v,p):=P_{C(p)}(v).
\end{equation}
Our aim is to prove local lipschitzness of the projection mapping $P:\ \hilbertH\times \setD \rightarrow \hilbertH$, given by \eqref{projection} at an arbitrary fixed $(\bar{v},\bar{p})$. The following example shows that even if the functions $g_i:\ \setD\rightarrow \hilbertH$, $i\in I_1\cup I_2$ are globally Lipschitz, the projection onto $C(p)$, $p\in \setD$, given by \eqref{multifunction_C}, may not be continuous (in the strong topology). For other examples see e.g. \cite{MR0388177}.
\begin{example}
	Let $p\in \mathbb{R}$, $\bar{p}=0$, $\bar{v}=(-1,-1)$ and
	\begin{equation*}
		C(p)=\left\{ x\in \mathbb{R}^2 \ \bigg|\  \begin{array}{l}
			\langle x \ |\ (-1+p,0) \rangle \leq 0 \\
			\langle x \ |\ (0,p) \rangle \leq 0 \\
		\end{array}  \right\}.
	\end{equation*}
	The projection of $v=(v_1,v_2)$ from a neighbourhood of $\bar{v}$ onto $C(p)$, for $p$ close to $\bar{p}$ is equal to
	\begin{equation*}
		P(p,v)=\left\{\begin{array}{ll}
			(0,0) & \text{if}\ -1<p<0 \\
			(0,v_2) & \text{if}\ 1>p\geq0 
		\end{array}\right. .
	\end{equation*}
	Hence, $P(\cdot,\cdot)$ is not continuous at $(\bar{p},\bar{v})$.
\end{example}

Our analysis is based on a recent results of \cite{full_stability_of_general_parametric_variational_systems}
concerning lipschitzness (and h\"olderness) of solutions to  
a class of parametric variational inclusions.

Essential part of our considerations is based on the relaxed constant rank constraint qualification (RCRCQ) introduced in \cite{om_relaxed_constant_rank_regularity_condition} and investigated in \cite{on_lipschitz_like_property,relation_between_the_consant_rank_Lu,parametric_nonlinear_programming_Minchenko}.
According to our knowledge, no result is known in the literature, in which this particular constraint qualification condition is used in the context of stability of solutions to parametric problems \eqref{problem:par} with $I_1\neq \emptyset$. 
Moreover, we assume only local lipschitzness of the right-hand-side functions $f_i$, $i\in I_1\cup I_2$ and left-hand-side functions $g_i$, $i\in I_1\cup I_2$.

Observe, that, in general, the existing continuity-type results for solutions of  problem \eqref{problem:par} are representation-dependent in the sense that e.g. MFCQ condition is representation-dependent. 
Observe that RCRCQ (Definition \ref{def:RCRCQ}) is also representation-dependent. We take this fact into account  by introducing the concept of equivalent representation (Definition \ref{def:equivalent}) and the concept of equivalent stable representation (Definition \ref{def:stability}). In Theorem \ref{theorem:necessary} we show that the under assumption (H1) the existence of a suitable  equivalent representation is necessary for the continuity of projections onto sets $C(p)$, $p\in \setD$, given by \eqref{multifunction_C}.

The organization of the paper is as follows. In Theorem \ref{theorem:Mordukhovic_2} of Section \ref{section:underlying_facts} we recall Theorem 6.5 of \cite{full_stability_of_general_parametric_variational_systems} in the form 
which corresponds to our settings. Theorem \ref{theorem:Mordukhovic_2} provides sufficient and necessary conditions for the estimate \eqref{lipschitzian_stability_of_projections} which is stronger than local lipschitzness of projection $P(\cdot,\cdot)$ (see \eqref{cond:II} of Theorem \ref{theorem:Mordukhovic_2}). For convenience of the reader we provide the proof of the sufficiency part of Theorem \ref{theorem:Mordukhovic_2}.
In Section \ref{section:RCRCQ} we recall the relaxed constant rank qualification (RCRCQ) and the results concerning Lipschitz-likeness of parametrized constrained sets $C(p)$, $p\in\setD$, given by \eqref{multifunction_C}. 
In Section \ref{section:rcrcq_and_Lagrange_multipliers} we investigate Lagrange multipliers of problem \eqref{M(v,p)} under RCRCQ. 
In Section \ref{section:stable_representations} we introduce the concept of equivalent stable representation of sets $C(p)$, $p\in \setD$, given by \eqref{multifunction_C}. Main results of this section are Corollary \ref{remark:RCRCQ_extended} and Theorem \ref{theorem:necessary}. 
Section \ref{section:main_results} contains main result of the present paper (Theorem \ref{theorem:main}) together with a number of corollaries referring to several particular cases of problem \eqref{M(v,p)}. Section \ref{section:conclusion} concludes.
\section{Underlying facts}\label{section:underlying_facts}

Finding $P(v,p)$, given by \eqref{projection}, amounts  to solving the  parametric variational inequality
\begin{equation}\tag{PV($v$,$p$)}\label{variational_ineq1}
	\text{find}\ x\in \hilbertH\ \text{s.t}\quad v\in x+ N(x;C(p)),
\end{equation}
where $N(x;C(p))$ stands for the normal cone (in the sense of convex analysis) to the set $C(p)$ at $x\in C(p)$ i.e.,
\begin{equation*}
	N(x;C(p)):=\{ h\in \hilbertH \ |\ \langle h \mid y-x\rangle \leq 0,\ \forall y \in C(p)  \}.
\end{equation*} 

Local lipschitzness of solutions to general parametric variational inequality has been recently investigated by Mordukhovich, Nghia and Pham in \cite[Theorem 6.5]{full_stability_of_general_parametric_variational_systems}.

We apply Theorem 6.5 of \cite{full_stability_of_general_parametric_variational_systems} to investigate conditions under which the mapping $P:\ \hilbertH\times \setD\rightarrow \hilbertH$ defined by \eqref{projection} is locally Lipschitz at a given $(\bar{v},\bar{p})\in \hilbertH\times \setD$. 

From the point of view of applications it is also interesting to investigate the particular case of problem $\eqref{M(v,p)}$ with $v\equiv \bar{v}$, i.e., 
\begin{align}\tag{M($p$)}\label{M(p)}
	\begin{aligned}
		&\min_{x\in \hilbertH}\, \frac{1}{2}\|x-\bar{v}\|^2,\\
		&\text{subject to}\ x\in C(p):=\left\{ x\in \hilbertH\ \bigg|\ 
		\begin{array}{ll}
			\langle x \ |\ g_i(p)\rangle = f_i(p), &i \in I_1,\\
			\langle x \ |\ g_i(p)\rangle \leq f_i(p), &i \in I_2
		\end{array}
		\right\},
	\end{aligned}
\end{align}
where $\bar{v}\in \hilbertH$, i.e. the problem $\eqref{M(p)}$ does not depend on parameter $v$.

When $g_i(p)\equiv g_i\in \hilbertH$, $i\in I_1\cup I_2$, the sets $C(p)$ take the form
\begin{equation}
	\label{fixed_normal_vectors}
	\hat{C}(p):=\left\{ x\in \hilbertH\ \bigg|\ 
	\begin{array}{ll}
		\langle x \ |\ g_i\rangle = f_i(p), &i \in I_1,\\
		\langle x \ |\ g_i\rangle \leq f_i(p), &i \in I_2
	\end{array}
	\right\}
\end{equation}
and the stability of the respective variational system $\eqref{variational_ineq1}$ has been investigated in \cite{lipschitz_continuity_of_solutions_of_variatioal_inequalities}.

For any multifunction ${\mathcal F}:\ X\rightrightarrows Y$ its domain and graph are defined as
\begin{align*}
	\mbox{dom}\,{\mathcal F}	&:=\{u\in X\ |\ {\mathcal F}(u)\neq\emptyset\},\\
	\gph{\mathcal F} &:= \{(u,y) \in X\times Y \mid y\in F(u) \}. 
\end{align*}

\begin{definition} \cite[Definition 1.40]{Mordukhovich_book_1} Let $X$, $Y$ be normed spaces.
	\label{definition_pseudo-Lipschitz} Let ${\mathcal F}:\ X\rightrightarrows Y$ be a multifunction with $\text{dom\, }{\mathcal F}\neq \emptyset$.
	Given $(\bar{u},\bar{y})\in \gph {\mathcal F}$, we say that ${\mathcal F}$ is \textit{Lipschitz-like} (pseudo-Lipschitz, has the Aubin property) around $(\bar{u},\bar{y})$ with modulus $\ell\geq 0$ if there are neighbourhoods $U(\bar{u})$ of $\bar{u}$ and $V(\bar{y})$ of $\bar{y}$ such that 
	\begin{equation*}
		{\mathcal F}(u_1)\cap V(\bar{y})\subset {\mathcal F}(u_2)+\ell\|u_1-u_2\|B(0,1)\quad \text{for all}\ u_1,u_2\in U(\bar{u}),
	\end{equation*}
	where $B(0,1)$ is the open unit ball in $Y$.
\end{definition}
Let  	   $\multifC:\ \setD\rightrightarrows \mathcal H$ be a set-valued mapping  defined as $\multifC (p):=C(p)$,
\begin{equation}\label{multifunction_C}
	C(p)=\left\{ x\in \hilbertH\ \bigg|\ 
	\begin{array}{ll}
		\langle x \ |\ g_i(p)\rangle = f_i(p), &i \in I_1,\\
		\langle x \ |\ g_i(p)\rangle \leq f_i(p), &i \in I_2
	\end{array}
	\right\},
\end{equation}
where $f_i:\ \setD\rightarrow \mathbb{R}$, $g_i:\ \setD\rightarrow\hilbertH$, $i\in I_1\cup I_2\neq \emptyset,$ $I_1=\emptyset \vee \{1,\dots,q\}$, $I_2=\emptyset \vee \{q+1,\dots,m\}$  are   locally Lipschitz  on $\setD$. The set $C(p)$ is the feasible solution  set of  problem \eqref{M(v,p)}.

In view of \cite[Example 6.4]{full_stability_of_general_parametric_variational_systems}, the fact \cite[Lemma 6.2]{full_stability_of_general_parametric_variational_systems}   
applied to problem \eqref{variational_ineq1} takes the following form.

\begin{proposition}\cite[Lemma 6.2]{full_stability_of_general_parametric_variational_systems} \label{proposition:holderian}
	Let $\bar{x}=P(\bar{v},\bar{p})$, $\bar{v}\in \hilbertH$, $\bar{p}\in \setD$. If $\multifC$ is Lipschitz-like around $(\bar{p},\bar{x})$, then there exist constants $\kappa_0,\ell^0>0$ and neighbourhoods $W(\bar{v})$, $U(\bar{p})$  that the estimate
	\begin{align}\label{condition:holderian_stability}
		\begin{aligned}
			\| (v_1-v_2)-2\kappa_0[P(v_1,p_1)-P(v_2,p_2)]\|\leq \|v_1-v_2\|+\ell^0\|p_1-p_2\|^{1/2}
		\end{aligned}		
	\end{align}
	holds for all $(v_1,p_1),(v_2,p_2)\in W(\bar{v})\times U(\bar{p})$.
\end{proposition}
Let us note, that in view of Lemma 6.2 of \cite{full_stability_of_general_parametric_variational_systems}, we have $\kappa_0=1-\lambda r$, where in our case $\lambda=1$ and $r=0$ (Lemma 5.2 of \cite{full_stability_of_general_parametric_variational_systems} remain true for $\multifR=r=0$), i.e., \eqref{condition:holderian_stability} takes the form 
\begin{align}\label{condition:holderian_stability2}
	\begin{aligned}
		\| (v_1-v_2)-2[P(v_1,p_1)-P(v_2,p_2)]\|\leq \|v_1-v_2\|+\ell^0\|p_1-p_2\|^{1/2}.
	\end{aligned}		
\end{align}

For  problems considered in the present paper, Theorem 6.5 of \cite{full_stability_of_general_parametric_variational_systems} takes the following form.

\begin{theorem}\cite[Theorem 6.5]{full_stability_of_general_parametric_variational_systems}\label{theorem:Mordukhovic_2}
	Let $\bar{p}\in \setD$, $\bar{v}\in \hilbertH$ and $\bar{x}=P(\bar{v},\bar{p})$.	Suppose that 
	\begin{enumerate}[(A)]
		\item\label{cond:A}
		$\multifC$ is Lipschitz-like around $(\bar{p},\bar{x})$.
	\end{enumerate}
	The following conditions are equivalent.
	\begin{enumerate}[(I)]
		\item\label{cond:I} The graphical subdifferential mapping $Gr:\ \setD\rightrightarrows\hilbertH\times \hilbertH$ defined as
		\begin{equation}\label{mulfiunction:K}
			Gr(p):=\{ (x,x^\prime)\ |\ x\in C(p),\ x^\prime \in N(x;C(p))  \}=\gph N(\cdot;C(p))
		\end{equation}
		is Lipschitz-like around $(\bar{p},\bar{x},\bar{v}-\bar{x})$.
		\item\label{cond:II} 	There exist  neighbourhoods $W(\bar{v})$, $U(\bar{p})$  such that the estimate
		\begin{align*}
			\begin{aligned}
				\| (v_1-v_2)-2[P(v_1,p_1)-P(v_2,p_2)]\|\leq \|v_1-v_2\|+\ell^0\|p_1-p_2\|
			\end{aligned}		
		\end{align*}
		holds for all $(v_1,p_1),(v_2,p_2)\in W(\bar{v})\times U(\bar{p})$ with some positive constant 
		$\ell^0$.
	\end{enumerate}
	
\end{theorem}

In view of applications we have in mind, for convenience of the reader, we provide the proof of (I) $\implies$ (II) of Theorem \ref{theorem:Mordukhovic_2}.

\begin{proof}
	By Proposition \ref{proposition:holderian}, there exist constant
	$\ell^0>0$  and neighbourhoods $V(\bar{v})$, $Q(\bar{p})$  such that the estimate
	\begin{align}\label{ineq:Holder}
		\begin{aligned}
			\| (v_1-v_2)-2[P(v_1,p_1)-P(v_2,p_2)]\|\leq \|v_1-v_2\|+\ell^0\|p_1-p_2\|^{1/2}
		\end{aligned}		
	\end{align}
	holds for all $(v_1,p_1),(v_2,p_2)\in V(\bar{v})\times Q(\bar{p})$. Moreover, by \eqref{cond:I}, there exist  neighbourhoods $Q_1(\bar{p})\subset Q(\bar{p})$, $U_1(\bar{x})$, $V_1(\bar{v}-\bar{x})$ for which $v+u\in V(\bar{v})$ whenever $(u,p,v)\in U_1(\bar{x})\times Q_1(\bar{p})\times V_1(\bar{v}-\bar{x})$ such that
	\begin{equation}\label{eq:Lipschitzian}
		\text{Gr}(p_1)\cap [U_1(\bar{x})\times V_1(\bar{v}-\bar{x})]\subset \text{Gr}(p_2)+\ell_{Gr} \|p_1-p_2\|B(0,1)
	\end{equation}
	holds for all $p_1,p_2\in Q_1(\bar{p})$, where $\ell_{Gr}>0$ is a constant (here $B(0,1)$ is the unit ball in $\hilbertH\times\hilbertH$).	By \eqref{condition:holderian_stability2}, we have
	\begin{align*}
		\|v_1-v_2\|+\ell^0\|p_1-p_2\|^{1/2}&\geq 	\| (v_1-v_2)-2[P(v_1,p_1)-P(v_2,p_2)]\|\\
		&\geq |2\|P(v_1,p_1)-P(v_2,p_2)\|-\|v_1-v_2\||\\
		& \geq 2\|P(v_1,p_1)-P(v_2,p_2)\|-\|v_1-v_2\|.
	\end{align*}
	Hence
	\begin{equation*}
		\|P(v_1,p_1)-P(v_2,p_2)\|\leq \|v_1-v_2\|+\frac{\ell^0}{2}\|p_1-p_2\|^{1/2}.
	\end{equation*}
	There exist neighbourhoods $U_2(\bar{x}), Q_2(\bar{p}),V_2(\bar{v})$, such that $U_2(\bar{x})\times Q_2(\bar{p})\times V_2(\bar{v})\subset U_1(\bar{x})\times Q_1(\bar{p})\times V(\bar{v})$ and $P(V_2(\bar{x}),Q_2(\bar{p}))\subset U_2(\bar{x})$, and 
	\begin{equation*}
		v-u=v-\bar{v}-(u-\bar{x})+\bar{v}-\bar{x}\subset V_1(\bar{v}-\bar{x})
	\end{equation*}
	for all $(u,p,v)\in U_2(\bar{x})\times Q_2(\bar{p})\times V_2(\bar{v})$.
	Now pick $(v_1,p_1),(v_2,p_2)\in V_2(\bar{v})\times Q_2(\bar{p})$ and define $u_1:=P(v_1,p_1)\in U_2(\bar{x})$ and $u_2:=P(v_2,p_2)\in U_2(\bar{x})$. Therefore, we have $v_1^\prime:=v_1-u_1\in N(u_1;C(p_1))\cap V_1(\bar{v}-\bar{x})$, i.e.,
	$(u_1,v_1^\prime)\in Gr(p_1)\cap (U_1(\bar{x})\times V_1(\bar{v}-\bar{x}))$. By \eqref{eq:Lipschitzian},  there is $(u,v)\in Gr(p_2)$ satisfying 
	\begin{equation*}
		\|u-u_1\|+\|v-v_1^\prime\|\leq \ell_{Gr}\|p_1-p_2\|.
	\end{equation*}
	Define $v^\prime:= u+v \in u + N(u;C(p_2))$, i.e., $u=P(v^\prime, C(p_2))$. Then
	\begin{align*}
		\|v^\prime -v_1\|= \| u+v - u_1 -v_1^\prime \|\leq \|v-v_1^\prime\| + \| u-u_1\|	\leq \ell_{Gr}\|p_1-p_2\|.
	\end{align*}
	Hence,  $V^\prime(\bar{v}) \subset V(\bar{v})$ by choosing $Q_2(\bar{p})$ sufficiently small. By \eqref{ineq:Holder}, for pairs $(v^\prime,p_2)$ and $(v_2,p_2)$ we have
	\begin{equation*}
		\|(v^\prime-v_2)-2(u-u_2)\|\leq \|v^\prime -v_2\|.
	\end{equation*}
	Hence, for any $(v_1,p_1),(v_2,p_2)\in V_2(\bar{v})\times Q_2(\bar{p})$,
	\begin{align*}
		\|(v_1-v_2)-2 (u_1-u_2)\|&\leq \|(v^\prime -v_2)-2 (u-u_2)\|+\|v^\prime -v_1 \| +2 \|u_1-u\|\\
		&\leq \|v^\prime -v_2 \| +\|v^\prime -v_1 \| +2 \|u_1-u\|\\
		&\leq \|v_1 -v_2 \| + \|v^\prime -v_1\| +  \|v^\prime -v_1 \| +2 \|u_1-u\|\\
		&\leq \|v_1 -v_2 \| +4 \ell_{Gr}\|p_1-p_2\|.
	\end{align*}
\end{proof}
\begin{remark}
	It follows from the proof that under assumptions of Theorem \ref{theorem:Mordukhovic_2} and condition \eqref{cond:I}, the  estimate in \eqref{cond:II} holds with constant $\ell^0=4\ell_{Gr}$.
\end{remark}

\begin{remark}\label{remark:equivalence_domain}
	Clearly, 
	$
	\text{dom}\, Gr=\{p\in{\mathcal G}\ |\ Gr(p)\neq\emptyset\}=\text{dom\,}\mathbb{C}
	$
	and
	$$
	Gr(p)=\{(x,0) \ |\ x\in\text{int}\, C(p)\}\cup\{(x,x')\ |\ x\in\text{bd}\, C(p),\ x'\in N(x; C(p))\neq\{0\} \}
	$$
	for  $p\in\text{dom}\, Gr$. Consequently, by taking $\bar{p}\in\text{dom\,}\mathbb{C}$, 
	$0\neq\bar{v}-\bar{x}\in N(\bar{x};C(\bar{p}))$ and a neighbourhood
	$V(0)$ in ${\mathcal H}$ such that $0\not\in \bar{v}-\bar{x}+V(0)$
	\begin{equation} 
		x\in (\bar{x}+V(0))\cap C(p)\wedge\  x'\in (\bar{v}-\bar{x}+V(0))\cap N(x;C(p))\ \ \Rightarrow
		x\in \text{bd\,} C(p)
	\end{equation}
	for $p$ close to $\bar{p}$.  
\end{remark}

In view of   Theorem \ref{theorem:Mordukhovic_2} to prove \eqref{cond:II} we need to show \eqref{cond:I} and the condition \eqref{cond:A}.
Condition \eqref{cond:A} was investigated in details in \cite{on_lipschitz_like_continuity_arxiv,on_lipschitz_like_property} and it is  discussed in Section \ref{section:RCRCQ}. Condition \eqref{cond:I} is proved  in Proposition \ref{prop:2ab2} in Section  \ref {section:main_results} with the help of  a number of  propositions proved in Section  \ref{section:rcrcq_and_Lagrange_multipliers}.

In the sequel we make an extensive use of the lower Kuratowski limit  for a multifunction ${\mathcal F}:\ \setD \rightrightarrows \hilbertH$ at $\bar{p}$ defined as 
\begin{equation*}
	\liminf_{p\rightarrow\bar{p},\ p\in \setD} {\mathcal F}(p):=\{ y\in \hilbertH\ |\ \forall\, p_k\rightarrow \bar{p},\ p_k\in \setD,\ \exists\, y_k\in {\mathcal F}(p_k)\quad y_k \rightarrow y\}.
\end{equation*}
Equivalently, $\bar{x}\in\liminf\limits_{p\rightarrow \bar{p},\ p\in \setD}{\mathcal F}(p)$ if for every neighbourhood $V(\bar{x})$ of $\bar{x}$
there exists a neighbourhood $U(\bar{p})$ of $\bar{p}$ such that $V(\bar{x})\cap {\mathcal F}(p)\neq\emptyset$ for  $p\in U(\bar{p})$.

The following   condition related to the lower Kuratowski limit is  necessary for the continuity of the projection
mapping $P$. 
\begin{proposition}
	Let $\bar{p}\in \setD$  and $\bar{v}\in\hilbertH$. If the mapping $P:\ \hilbertH \times \setD \rightarrow \hilbertH$ given by \eqref{projection}
	is continuous at $(\bar{v},\bar{p})\in \hilbertH\times \setD$ with $\bar{x}:=P(\bar{v},\bar{p})\in C(\bar{p})$, then 
	$\bar{x}\in\liminf\limits_{p\rightarrow \bar{p},\ p\in \setD}\multifC(p)$.
	
\end{proposition}
\begin{proof}
	Suppose, by contradiction, that $\bar{x}\not\in\liminf\limits_{p\rightarrow \bar{p},\ p\in \setD}\multifC(p)$ . By definition, there exists a neighbourhood $V(\bar{x})$ of $\bar{x}$ such that in  every neighbourhood $U(\bar{p})$ of 
	$\bar{p}$  there exists $p_{U}\in U(\bar{p})$ satisfying
	$$
	V(\bar{x})\cap C(p_{U})=\emptyset.
	$$
	In consequence,
	$P(\bar{v},p_{U})\not\in V(\bar{x}),$
	which contradicts the continuity of $P$ at $(\bar{v},\bar{p})$.
\end{proof}

\section{RCRCQ and Lipschitz-likeness of the set-valued mapping $\multifC$}
\label{section:RCRCQ}	
In this section we discuss Lipschitz-likeness of set-valued mapping 	   $\multifC:\ \setD\rightrightarrows \mathcal H$  defined by \eqref{multifunction_C}. 

For any $(p,x)\in \setD\times \hilbertH$, let $I_{p}(x):= \{ i\in I_1\cup I_2\ |\ \langle x\ |\ g_i(p)\rangle =f_i(p)  \}$ denote the active index set for $p\in \setD$ at $x\in \hilbertH$.

In our main results (Proposition \ref{lemma:indices_RCRCQ} , Theorem \ref{theorem:main}) we use the relaxed constant rank constraint qualification  as defined in \cite[Definition 4]{on_lipschitz_like_property}.
\begin{definition}\label{def:RCRCQ}
	The \textit{relaxed constant rank constraint qualification} (RCRCQ)  for multifunction $\multifC$ is satisfied at $(\bar{p},\bar{x})$, $\bar{x}\in C(\bar{p})$, if there exists a neighbourhood $U(\bar{p})$ of $\bar{p}$ such that, for any index set $J$, $I_1\subset J\subset I_{\bar{p}}(\bar{x})$, for every $p\in U(\bar{p})$ the system of vectors $\{ g_i(p), i\in J  \}$ has constant rank, i.e.,
	\begin{equation}\label{eq:RCRCQ}
		\text{rank}\{g_i(p),i\in J\}=\text{rank}\{g_i(\bar{p}),i\in J\}\quad \text{for all }p\in U(\bar{p}).
	\end{equation} 
\end{definition}
Clearly, condition \eqref{eq:RCRCQ} does not depend upon  $x$ from a neighbourhood of $\bar{x}$.

Proposition \ref{propostion:I_1-linearly_independent} says that we can always represent equivalently the set $C(p)$ given by \eqref{multifunction_C} in a neighbourhood of $\bar{p}\in \setD$ in the way that normal vectors of equality constraints are linearly independent. 
A finite-dimensional analogue of Proposition \ref{propostion:I_1-linearly_independent} has been established in \cite[Lemma 2.2]{relation_between_the_consant_rank_Lu}. 

\begin{proposition}\cite[Proposition 11]{on_lipschitz_like_property}\label{propostion:I_1-linearly_independent}
	Let $\bar{p}\in \setD$. Assume RCRCQ holds at $(\bar{p},\bar{x})$, $\bar{x}\in C(\bar{p})$ for multifunction $\multifC$ and $C(p)\neq \emptyset$ for $p\in U_0(\bar{p})$. There exists a neighbourhood $U(\bar{p})$ such that for all $p\in U(\bar{p})$
	\begin{align*}
		&\{ x\ |\ \langle x \ |\ g_i(p)\rangle = f_i(p),\ i\in I_1,\ \langle x \ |\ g_i(p)\rangle \leq  f_i(p),\ i\in I_2 \}\\
		&=\{ x\ |\ \langle x \ |\ g_i(p)\rangle = f_i(p),\ i\in I_1^\prime,\ \langle x \ |\ g_i(p)\rangle \leq  f_i(p),\ i\in I_2 \},
	\end{align*}
	where $I_1^\prime\subset I_1$, $|I_1^\prime|=\text{rank} \{g_i(\bar{p}),\ i \in I_1 \}$ and $g_i(p)$, $i\in I_1^\prime$ are linearly independent. 
\end{proposition}

In view of Proposition \ref{propostion:I_1-linearly_independent}, in the sequel we assume that for any $\bar{p} \in \setD$,  $g_i(p)$, $i\in I_1$ are linearly independent in some neighbourhood $U(\bar{p})$.

\begin{remark}
	\label{remark_right_handside}  Let us note that for the set-valued mapping  $\hat{\multifC}:\ \setD\rightrightarrows \mathcal H$, 
	$\hat{\multifC} (p):=\hat{C}(p)$, with $\hat{C}(p)$ defined by
	\eqref{fixed_normal_vectors}, the relaxed constant rank constraint qualification condition RCRCQ is satisfied at any $(p,x)\in\text{gph}\hat{\multifC}$.
\end{remark}

In the case of absence of equality constraints  in \eqref{multifunction_C} the condition RCRCQ is equivalent to constant rank constraint qualification (CRCQ) (see \cite{ANDREANI2014,directional_derivative_Janin,implications_of_the_consant_rank_Lu}) which has been already used in \cite{full_lipschitzian_and_holderian_stability_Mordukhovich} in proving lipschitzness of projections.

The following theorem has been proved in \cite{on_lipschitz_like_property}.
\begin{theorem}(\cite[Theorem 9]{on_lipschitz_like_property})\label{theorem:main_section_RCRCQ}
	Let $\mathcal H$ be a Hilbert space, $\setD\subset \hilbertG$ be a subset of a normed space $\hilbertG$ and let $\multifC:\ \setD\rightrightarrows \hilbertH$ be given by \eqref{multifunction_C}. Assume RCRCQ is satisfied at 
	$(\bar{p},\bar{x})\in \gph \multifC$ and $\bar{x}\in \liminf\limits_{p\rightarrow \bar{p},\ p\in \setD} \multifC(p)$. Then $\multifC$ is Lipschitz-like at $(\bar{p},\bar{x})$.
\end{theorem}
By using Proposition \ref{proposition:holderian} and Theorem \ref{theorem:main_section_RCRCQ}, we obtain the following H\"older type estimate for solutions to problem \eqref{M(v,p)}.
\begin{corollary}\label{corollay:holderian}
	Let $\mathcal H$ be a Hilbert space, $\setD\subset \hilbertG$  be a subset of a normed space $\hilbertG$ and let $\multifC:\ \setD\rightrightarrows \hilbertH$ be given by \eqref{multifunction_C}. Assume RCRCQ is satisfied at a point $(\bar{p},\bar{x})\in \gph \multifC$ and $\bar{x}\in \liminf\limits_{p\rightarrow \bar{p},\ p\in \setD} \multifC(p)$. Then there exist constant $\ell^0>0$ and neighbourhoods $W(\bar{v})$, $U(\bar{p})$  that the estimate
	\begin{align*}
		\begin{aligned}
			\| (v_1-v_2)-2[P(v_1,p_1)-P(v_2,p_2)]\|\leq \|v_1-v_2\|+\ell^0\|p_1-p_2\|^{1/2}
		\end{aligned}		
	\end{align*}
	holds for all $(v_1,p_1),(v_2,p_2)\in W(\bar{v})\times U(\bar{p})$.
\end{corollary}
In view of Remark \ref{remark_right_handside}, the following result is an immediate consequence of Theorem \ref{theorem:main_section_RCRCQ}.

\begin{theorem}\label{theorem:main_section_RCRCQ_2}
	Let $\mathcal H$ be a Hilbert space, $\setD\subset \hilbertG$  be a subset of a normed space $\hilbertG$ and let $\hat{\multifC}:\ \setD\rightrightarrows \hilbertH$,
	$\hat{\multifC} (p):=\hat{C}(p)$, with $\hat{C}(p)$ given by
	\eqref{fixed_normal_vectors}. Assume that $(\bar{p},\bar{x})\in \gph \hat{\multifC}$ and $\bar{x}\in \liminf\limits_{p\rightarrow \bar{p},\ p\in \setD} \hat{\multifC}(p)$. Then $\hat{\multifC}$ is Lipschitz-like at $(\bar{p},\bar{x})$.
\end{theorem}

The Lipschitz-likeness of $\hat{\multifC}$ has already been investigated in the finite-dimensional case in \cite{aubin_criterion_for_metric_regularity} with the help of the Mangasarian-Fromowitz constraint qualification MFCQ.

\begin{definition}
	\label{def_Mangasarian_Fromowitz}
	We say that the \textit{Mangasarian-Fromowitz constraint qualification} (MFCQ) holds for $C(\bar{p})$ at $\bar{x}\in C(\bar{p})$ if vectors $g_{i}(\bar{p})$, $i\in I_{1}$ are  linearly independent and there exists $h\in{\mathcal H}$ such that
	$$
	\langle g_{i}(\bar{p})|h\rangle =0\ \ i\in I_{1},\ \ \langle g_{i}(\bar{p})|h\rangle <0\ \ i\in I_{\bar{p}}(\bar{x})
	$$
\end{definition}

The following fact relates the Mangasarian-Fromowitz constraint qualification MFCQ to the lower Kuratowski limit of the set-valued mapping
$\bar{\multifC}:\ \setD\rightrightarrows \hilbertH$,
$\bar{\multifC} (p):=\bar{C}(p)$, with $\bar{C}(p)$ given by
\eqref{multifunction_C} with $I_{1}=\emptyset$, i.e.
\begin{equation*}
	\bar{C}(p):=\left\{ x\in \hilbertH\ \bigg|\ 
	\begin{array}{ll}
		\langle x \ |\ g_i(p)\rangle \leq f_i(p), &i \in I_2
	\end{array}
	\right\}.
\end{equation*}

\begin{proposition}
	\label{prop:MFCQ-Kuratowski}
	If MFCQ holds for  $\bar{C}(\bar{p})$ at $\bar{x}\in \bar{C}(\bar{p})$,
	then $\bar{x}\in\liminf\limits_{p\rightarrow\bar{p},\ p\in \setD} \bar{C}(p)$.
\end{proposition}
\begin{proof} 
	By  MFCQ, there exists $h\in{\mathcal H}$ such that
	$$
	\langle g_{i}(\bar{p})|h\rangle <0\ \ \text{for}\ \ i\in I_{\bar{p}}(\bar{x}).
	$$
	Let $V(\bar{x})$ be any neighbourhood of $\bar{x}$. There exists $\alpha>0$ such that
	$$
	\langle g_{i}(\bar{p})|\bar{x}+\alpha h\rangle<f_{i}(\bar{p})\ \ i\in I_{2}
	$$
	and $\bar{x}+\alpha h\in V(\bar{x})$. 
	Since the functions $g_{i},f_{i}$, $i\in I_1\cup I_2$ are assumed to be locally Lipschitz at $\bar{p}$ there exists a neighbourhood $U_{i}(\bar{p})$ of $\bar{p}$
	such that
	\begin{align}
		\label{lip}
		\begin{aligned}
			\langle g_{i}(p)-g_{i}(\bar{p})\ |\ \bar{x}+\alpha h\rangle\leq \ell_{g_{i}}\|p-\bar{p}\|\|\bar{x}+\alpha\cdot h\|\ \ p\in U_{i}(\bar{p}),\ i\in I_2,\\
			f_{i}(\bar{p})-\ell_{f_i}\|p-\bar{p}\|\leq f_{i}(p)\  \ \ p\in U_{i}(\bar{p}),\ i\in I_2,
		\end{aligned}
	\end{align}
	where $\ell_{f_i}$, $\ell_{g_i}$ are locally Lipschitz constants of functions $f_i$, $g_i$, $i\in I_1\cup I_2$ at $\bar{p}$, respectively.
	
	Take $\varepsilon>0$, $\varepsilon<f_{i}(\bar{p})-\langle g_{i}(\bar{p})|\bar{x}+\alpha h\rangle>0.$ By shrinking  the neighbourhood $U_{i}(\bar{p})$, $i\in I_2$, we can assume that
	$$
	\ell_{g_{i}}\|p-\bar{p}\|\|\bar{x}+\alpha\cdot h\|+\ell_{f_i}\|p-\bar{p}\|<\varepsilon.
	$$
	Consequently, for $p\in U_{i}(\bar{p})$, $i\in I_2$ we have
	$$
	\langle g_{i}(\bar{p})|\bar{x}+\alpha h\rangle+\ell_{g_{i}}\|p-\bar{p}\|\|\bar{x}+\alpha\cdot h\|+\ell_{f_i}\|p-\bar{p}\|\le \langle g_{i}(\bar{p})|\bar{x}+\alpha h\rangle +\varepsilon \le f_{i}(\bar{p})
	$$
	By  \eqref{lip},
	$$
	\begin{array}{l} 
	\langle g_{i}(p)|\bar{x}+\alpha h\rangle\le f_{i}(\bar{p})-\ell_{f_i}\|p-\bar{p}\|\le f_{i}(p),\quad i\in I_2.
	\end{array}
	$$
	By taking $U(\bar{p})=\bigcap_{i\in I_2}U_{i}(\bar{p})$,
	we obtain the assertion.	
\end{proof}
The following example shows that MFCQ is not a necessary condition for Lipschitz continuity of projection of $v$ onto $C(p)$, $p\in \setD$, given in \eqref{multifunction_C}.
\begin{example}
	Let $p\in B((0,0),1)\subset \mathbb{R}^2$, $\bar{p}=(0,0)$, $\bar{v}=(1,1)$ and
	\begin{equation*}
		C(p)=\left\{ x\in \mathbb{R}^2 \ \bigg|\  \begin{array}{l}
			\langle x \ |\ (1,0)-p \rangle - \langle p \ |\ (1,0)-p\rangle \leq 0 \\
			\langle x \ |\ (0,1)-p \rangle - \langle p \ |\ (0,1)-p\rangle \leq 0 \\
			\langle x \ |\ (-1,-1)-p \rangle - \langle p \ |\ (-1,-1)-p\rangle \leq 0 
		\end{array}  \right\}
	\end{equation*}
	Then for all $p\in B((0,0),1)$ we have $C(p)=\{p\}$. Hence $P(v,p)=p$ for $p\in B((0,0),1)$ and for any $v\in \mathbb{R}^2$. Hence, $P(\cdot,\cdot)$ is locally Lipschitz in a neighbourhood of $(\bar{v},\bar{p})$ and MFCQ is not satisfied at $P(\bar{v},\bar{p})$.
\end{example}

\section{RCRCQ and Lagrange multipliers}\label{section:rcrcq_and_Lagrange_multipliers}
In this section we investigate properties of Lagrange multipliers of problem \eqref{M(v,p)} under RCRCQ condition.
We start with the following elementary observation.

\begin{remark}\label{remark:inclusion}
	Let $\bar{p}\in \setD$, $\bar{x}\in C(\bar{p})$, where $\multifC$ is given by \eqref{multifunction_C}.	
	By the continuity of $\funf_i$, $\fung_i$, $i\in I_1\cup I_2$, at $\bar{p}$ and the continuity of the inner product, there exist a neighbourhood $U(\bar{p})$ and a neighbourhood $V(0)$ of $0\in \hilbertH$,  such that 
	\begin{align*}
		\begin{aligned}
			&\langle x \ |\ \fung_i(p) \rangle <f_i(p) \quad i\in I_2\setminus I_{\bar{p}}(\bar{x})\\
			&\langle x \ |\ \fung_i(p) \rangle \leq  f_i(p) \quad i\in I_{\bar{p}}(\bar{x})
		\end{aligned}
	\end{align*}
	for all $p\in U_2(\bar{p})$, $x\in (\bar{x}+V(0))\cap C(p)$. Hence,  $I_p(x)\subset I_{\bar{p}}(\bar{x})$ for $p\in U(\bar{p})$ and $x\in (\bar{x}+V(0))\cap C(p)$.
\end{remark}
In Proposition \ref{lemma:indices_RCRCQ} we investigate representations of
elements of $N(\bar{x}; C(\bar{p}))$ of the form \eqref{eq:representation1} in a neighbourhood of $(\bar{p},\bar{x},\bar{v}-\bar{x})$, where $\bar{v}-\bar{x}\in N(\bar{x}; C(\bar{p}))$.
Moreover, we prove that for all $p$ close to $\bar{p}$, for all $x\in C(p)$ close to $\bar{x}$, for all $x'\in N(x;C(p))$ close to $\bar{v}-\bar{x}$, there exists a representation  
$$
x'=\sum_{i \in {I_{\bar{p}}}(\bar{x})} \tilde{\lambda}_{i}g_{i}(p)
$$
where the function $\lambda:\ U(\bar{p})\times (\bar{x}+V(0))\times (\bar{v}-\bar{x}+V(0)) \rightarrow\mathbb{R}^{|I_{\bar{p}}(\bar{x})|}$,
$$
\lambda(p,x,x^\prime):=\left\{\begin{array}{lcl}
(\tilde{\lambda}_{i})_{i\in I_{\bar{p}}(\bar{x})}& \text{if}& (p,x,x^\prime)\neq(\bar{p},\bar{x},\bar{v}-\bar{x}),\\
(\bar{\lambda}_{i})_{i\in I_{\bar{p}}(\bar{x})}& \text{if}& (p,x,x^\prime)=(\bar{p},\bar{x},\bar{v}-\bar{x})
\end{array}\right.
$$
is continuous at $(\bar{p},\bar{x},\bar{v}-\bar{x})$. 

\begin{proposition}\label{lemma:indices_RCRCQ} 
	Let $\bar{p}\in \setD$. 
	Suppose that $\bar{v}\notin C(\bar{p})$, $\bar{x}\in C(\bar{p})$, $I_{\bar{p}}(\bar{x})\neq \emptyset$.\footnote{With this assumption we limit our attention to point $\bar{x}$ which lay on the boundary of $C(\bar{p})$, $\bar{x}\in \text{bd}\, C(\bar{p})$, where $\text{bd}$ denotes the boundary of a set.} Assume that RCRCQ holds at $\bar{p}$ (with a neighbourhood $U_0(\bar{p})$) for multifunction $\multifC$, given by \eqref{multifunction_C}, and $C(p)\neq \emptyset$ for $p\in U_0(\bar{p})$.
	Let
	\begin{equation}\label{eq:representation1}
		\bar{v}-\bar{x}=\sum_{i \in I_1\cup \bar{K}} \bar{\lambda}_i \fung_i (\bar{p}),\ \text{where}\ \bar{\lambda}_i>0,\ i\in \bar{K}\subset I_{\bar{p}}(\bar{x})\cap I_2
	\end{equation}  and  $g_i(\bar{p})$, $i\in I_1\cup \bar{K}$ are linearly independent\footnote{We admit $K=\emptyset$.}. Then the following conditions hold.
	\begin{enumerate}[(i)]
		\item\label{cond:i} There exist  neighbourhoods $U(\bar{p})$, $V(0)$  such that for any $p\in U(\bar{p})$ and any $(x,x^\prime)\in \mbox{gph} N(\cdot;C(p))\cap (\bar{x}+V(0),\bar{v}-\bar{x}+V(0))$ there exists ${L}\subset (I_{p}(x)\cap I_2)\setminus \bar{K}\subset (I_{\bar{p}}(\bar{x})\cap I_2)\setminus \bar{K}$\footnote{We admit $L=\emptyset$.} such that the element $x^\prime$ can be represented as 
		\begin{align}\label{eq:representation2}
			\begin{aligned}
				& x^\prime = \sum_{i \in I_1} \lambda_i \fung_i(p)+\sum_{i \in \bar{K} } \lambda_i \fung_i(p)+\sum_{i \in L } \lambda_i \fung_i(p),\\
				&  \lambda_i> 0,\ i \in \bar{K},\ \lambda_i\geq 0,\ i\in  L,
			\end{aligned}
		\end{align}
		where   $g_i(p)$, $i\in I_1\cup \bar{K}\cup L$ are linearly independent.
		\item\label{cond:ii}  For any $\varepsilon>0$ one can choose in (i) neighbourhoods $U(\bar{p})$, $V(0)$  such that in the representation \eqref{eq:representation2} we have 
		\begin{align*}
			\begin{aligned}
				& \bar{\lambda}_i-\varepsilon\leq {\lambda}_i \leq \bar{\lambda}_i+\varepsilon\quad \forall i \in I_1\cup\bar{K},\\ 
				& 0\leq {\lambda}_i \leq \varepsilon\quad \forall i \in {L}.
			\end{aligned}
		\end{align*}
	\end{enumerate}
\end{proposition}

\begin{proof} 
	Since $\fung_i(\bar{p})$, $i\in I_1\cup \bar{K}$ are linearly independent and $\fung_i:\ \setD\rightarrow \hilbertH$, $i\in I_1\cup \bar{K}$ are continuous at $\bar{p}$, by Lemma \ref{lemma:linear_independent} (see Appendix), 
	there exists a neighbourhood $U_1(\bar{p})$ such that $\fung_i(p)$, $i\in I\cup \bar{K}$, $p\in U_1(\bar{p})$ are linearly independent.
	
	By Remark \ref{remark:equivalence_domain} and Remark \ref{remark:inclusion}, there exist  neighbourhoods $U_2(\bar{p})\subset U_1(\bar{p})$ and $V_1(0)$, such that for all $p\in U_2(\bar{p})$ if $(x,x^\prime)\in \mbox{gph} N(\cdot;C(p))\cap (\bar{x}+V_1(0),\bar{v}-\bar{x}+V_1(0))$, then $I_p(x)\subset I_{\bar{p}}(\bar{x})$, $x\in \text{bd}\, C(p)$ and $x^\prime \neq 0$.	
	Moreover, by 
	\cite[Theorem 6.40]{deutsch2001best}\footnote{Let us note that each equality constraint in the set $C(\cdot)$ can be represented as two inequalities, namely  $\langle x \mid g(p) \rangle = f_i(p) \iff \langle x \mid g(p) \rangle \leq f_i(p) \wedge \langle x \mid -g(p) \rangle \leq f_i(p)$ for any $x\in \hilbertH$, $p\in \setD$ .}
	we have 
	\begin{align}\label{representation:x_prime}
		\begin{aligned}
			&x^\prime = \sum_{i \in I_{p}(x)}\lambda_i \fung_i(p),\quad 
			\lambda_i\geq 0,\ i\in I_p(x)\cap I_2.
		\end{aligned}
	\end{align}

	Now, by contrary suppose that the assertion of the proposition does not hold, i.e., 
	there exist sequences $p_n\rightarrow \bar{p}$, $x_n\rightarrow\bar{x}$, $x_n\in C(p_n)$, $x_n^\prime\rightarrow\bar{v}-\bar{x}$, $x_n^\prime \in N(x_n;C(p_n))$ such that 
	\begin{equation}\label{formula:contradiction}
		\forall_{n\in \mathbb{N}}\quad x_n^\prime \ \text{can not be represented in the form \eqref{eq:representation2}}.
	\end{equation}
	By \eqref{representation:x_prime}, for all $n\in \mathbb{N}$, sufficiently large, $x_n^\prime$ can be represented in the form
	\begin{equation}\label{respresntation:normal}
		x_n^\prime = \sum_{i\in I_1}\lambda_i^n g_i(p_n)+\sum_{i \in I_{p_n}(x_n)\cap I_2} \lambda_i^n g_i(p_n),
	\end{equation}
	where $\lambda_i^n\geq 0$, $i\in I_{p_n}(x_n)\cap I_2$. We can rewrite \eqref{respresntation:normal} as 
	\begin{align*}
		x_n^\prime =& \sum_{i\in I_1}\lambda_i^n g_i(p_n)+
		\sum_{i\in I_{p_n}(x_n)\cap \bar{K}} \lambda_i^n g_i(p_n)+\sum_{i\in (I_{p_n}(x_n)\setminus \bar{K})\cap I_2} \lambda_i^n g_i(p_n),
	\end{align*} 
	and, by putting $\lambda_i^n=0$ for $i \in \bar{K}\setminus I_{p_n}(x_n)$, $n\in \mathbb{N}$, we get
	\begin{equation*}
		x_n^\prime = \sum_{i\in I_1}\lambda_i^n g_i(p_n)+\sum_{i\in \bar{K}} \lambda_i^n g_i(p_n)+\sum_{i\in (I_{p_n}(x_n)\setminus \bar{K})\cap I_2} \lambda_i^n g_i(p_n),
	\end{equation*} 
	where $\lambda_i^n\geq 0$, $i\in (I_{p_n}(x_n)\cap I_2)\cup \bar{K}$.
	By Lemma \ref{lemma:positive_combination2} (see Appendix), for all $n\in \mathbb{N}$, sufficiently large, there exist $\hat{I}_2^n\subset (I_{p_n}(x_n)\setminus \bar{K})\cap I_2$ and $\tilde{\lambda}_i^n$, $i\in I_1\cup \bar{K}\cup \hat{I}_2^n$, $\tilde{\lambda}_i^n\in \mathbb{R}$,\ $i\in I_1\cup \bar{K}$, $\tilde{\lambda}_i^n>0$, $i\in \hat{I}_2^n$ such that
	\begin{equation*}
		x_n^\prime = \sum_{i\in I_1}\tilde{\lambda}_i^n g_i(p_n)+\sum_{i\in \bar{K}} \tilde{\lambda}_i^n g_i(p_n)+\sum_{i\in \hat{I}_2^n} \tilde{\lambda}_i^n g_i(p_n),
	\end{equation*} 	
	where $\tilde{\lambda}_i\in \mathbb{R}$, $i\in I_1\cup \bar{K}$, $\tilde{\lambda}_i> 0 $, $i\in \hat{I}_2^n$ and $g_i(p_n)$, $i\in I_1\cup \bar{K}\cup \hat{I}_2^n$ are linearly independent.
	
	By passing to a subsequence, if necessary, we can assume that  $\hat{I}_2^n=:I_2^\prime$ and 
	\begin{equation}\label{respresntation:normal4}
		x_n^\prime = \sum_{i\in I_1}\tilde{\lambda}_i^n g_i(p_n)+\sum_{i\in \bar{K}} \tilde{\lambda}_i^n g_i(p_n)+\sum_{i\in I_2^\prime} \tilde{\lambda}_i^n g_i(p_n),
	\end{equation}
	where $\tilde{\lambda}_i^n$, $\tilde{\lambda}_i^n\in \mathbb{R}$, $i\in I_1\cup \bar{K}$, $\tilde{\lambda}_i^n> 0$, $i\in I_2^\prime$ and $g_i(p_n)$, $i\in I_1\cup \bar{K}\cup I_2^\prime$ are linearly independent. 
	
	By \eqref{formula:contradiction}, it must be $\tilde{\lambda}_i^n\leq 0$ for some $i_n\in \bar{K}$. Passing again to the subsequence in \eqref{respresntation:normal4}, if necessary, we conclude that there exists $i\in \bar{K}$ such that $\tilde{\lambda}_i^n\leq 0$. 
	
	On the other hand, by Lemma \ref{lemma:gramschmidt} (see Appendix), we have $\tilde{\lambda}_i^n\rightarrow \bar{\lambda}_i>0$, $i\in \bar{K}$, which leads to a contradiction. This proves \eqref{cond:i}.
	
	To prove \eqref{cond:ii} suppose there exist $\varepsilon>0$  and a sequence $\{i_n\}_{n\in\mathbb{N}}\subset I\cup \bar{K}\cup I_2^\prime$, such that in the representation \eqref{respresntation:normal4} for each $n\in \mathbb{N}$ one of the following holds:
	\begin{enumerate}
		\item $\tilde{\lambda}_{i_n}>\bar{\lambda}_i+\varepsilon$ and $i_n\in I_1\cup \bar{K}$,
		\item $\tilde{\lambda}_{i_n}<\bar{\lambda}_i-\varepsilon$ and $i_n\in I_1\cup \bar{K}$,
		\item $\tilde{\lambda}_{i_n}>\varepsilon$ and $i_n\in I_2^\prime$.
	\end{enumerate}
	By taking a subsequence of $\{x_n^\prime\}_{n\in\mathbb{N}}$, if necessary, one can  assume that  only one of the cases	 1., 2., 3. holds for all $n\in \mathbb{N}$.  On the other hand, by Lemma \ref{lemma:gramschmidt} (see Appendix), we have $\tilde{\lambda}_i^n\rightarrow \bar{\lambda}_i$, $i\in I_1\cup \bar{K}$ and $\tilde{\lambda}_i^n\rightarrow 0$, $i\in I_2^\prime$, which leads to a contradiction. This proves \eqref{cond:ii}.
\end{proof}

Recall that $\mbox{gph} N(\cdot;C(p))= \text{Gr}(p)$, $p\in \setD$.
It is clear that even if an element $\bar{v}-\bar{x}$ from the normal cone $N(\bar{x}; C(\bar{p}))$ has a unique representation  as a combination of some vectors $g_{i}(\bar{p})$, $i\in I_{\bar{p}}(\bar{x})$,  then, in a neighbourhood of $(\bar{p},\bar{x},\bar{v}-\bar{x})$, where $\bar{v}-\bar{x}\in N(\bar{x};C(\bar{p}))$, the elements $x^\prime\in N(x; C(p))$ may not have unique representations in terms of combinations of vectors
$g_{i}(p)$, $i\in I_{p}(x)$.

The example below illustrates the situation when the representation of $\bar{v}-\bar{x}$ is not  
unique.
\begin{example}\label{example:representation_not_unique}
	Let $p\in \mathbb{R}$, $\bar{v}=(0,1)$ and
	\begin{equation*}
		C(p)=\left\{ x\in \mathbb{R}^2 \bigg| \begin{array}{l}
			\langle x \ | \ (0,1) \rangle \leq 0\\
			\langle x - (|p|,0) \ | \ (1,0) \rangle \leq 0\\
			\langle x - (-|p|,0) \ |\ (-1,0) \rangle \leq 0
		\end{array}  \right\}
	\end{equation*}
	In this case $P(\bar{v},p)=(0,0):=\bar{x}$ for all $p\in \mathbb{R}$. However, for $\bar{p}=0$ we have $I_{\bar{p}}(\bar{x})=\{1,2,3\}$ and
	\begin{equation*} 
		\bar{v}-\bar{x}=1\cdot (0,1)+0\cdot(1,0)=1\cdot (0,1)+0\cdot(-1,0)
	\end{equation*}
	and for any $p\neq 0$, $x=(p,0)$, $x^\prime \in N(x;C(p))$ we have
	\begin{equation*}
		\begin{array}{lll}
			x^\prime = \lambda_1 \cdot (0,1)+\lambda_2 (1,0),& \lambda_1,\ \lambda_2\geq 0, & p>0\\
			x^\prime = \lambda_1 \cdot (0,1)+\lambda_2 (-1,0),& \lambda_1,\ \lambda_2\geq 0,& p<0.
		\end{array}
	\end{equation*}	
\end{example}

In the proposition below we show, that under assumptions appearing in Theorem \ref{theorem:main} we have
$\gph N(\cdot; C(p))\cap V(\bar{x},\bar{v}-\bar{x})\neq \emptyset$ for $(p,x,x^\prime)$ in some neighbourhood of  $(\bar{p},\bar{x},\bar{v}-\bar{x})$.

\begin{proposition}
	Let $\bar{p}\in \setD$, $\bar{v}\notin C(\bar{p})$ and $\bar{x}=P(\bar{v},\bar{p})$. Assume that RCRCQ holds at $(\bar{p},\bar{x})$ (with a neighbourhood $U_0(\bar{p})$) for multifunction $\multifC$, given by \eqref{multifunction_C}, and $\bar{x}\in \liminf\limits_{p\rightarrow\bar{p},\ p\in \setD} \multifC(p)$. Then
	\begin{equation*}
		\forall V(\bar{x},\bar{v}-\bar{x}) \ \exists U(\bar{p}) \ \forall p \in U(\bar{p})\quad \gph N(\cdot; C(p))\cap V(\bar{x},\bar{v}-\bar{x})\neq \emptyset,		
	\end{equation*}
	i.e., $(\bar{x},\bar{v}-\bar{x})\in \liminf\limits_{p\rightarrow\bar{p},\ p\in \setD} \gph N(\cdot, C(p))=\liminf\limits_{p\rightarrow\bar{p},\ p\in \setD} Gr(p)$.
\end{proposition}
\begin{proof}
	By contrary, suppose  that
	there exist a neighbourhood $V(\bar{x},\bar{v}-\bar{x})$ and a sequence $p_n\rightarrow \bar{p}$ such that 	$$\gph N(\cdot; C(p_n))\cap V(\bar{x},\bar{v}-\bar{x})= \emptyset.$$ Without loss of generality we may assume that $p_n\in U_0(\bar{p})$. 
	Let $x_n:=P(\bar{v},p_n)$. 
	Then, by Corollary \ref{corollay:holderian}, we have $x_n\rightarrow \bar{x}$. Moreover, $x_n^\prime :=\bar{v}-x_n\in N(x_n;C(p_n))$ and $x_n^\prime\rightarrow \bar{v}-\bar{x}$. Thus, $(x_n,x_n^\prime)\in \gph N(\cdot; C(p_n))\cap V(\bar{x},\bar{v}-\bar{x})$ for large $n$.
\end{proof}

Let us recall that the set of Lagrange multipliers associated with problem \eqref{M(v,p)} is defined as
\begin{equation*}
	\Lambda_{v}(p,x):=\{ \lambda\in\mathbb{R}^m\ |\ v-x=\sum_{i\in I_1\cup I_2} \lambda_i g_i(p)  \text{ where, for } i \in I_2,\ \lambda_i\geq 0,\ \lambda_i g_i(p)=0   \}
\end{equation*}
and for $M>0$ let
\begin{equation*}
	\Lambda_{v}^M(p,x):=\{\lambda \in \Lambda_{v}(p,x) \mid \sum_{i\in I_1\cup I_2}|\lambda_i| \leq M \}.
\end{equation*}
In proposition below we show that $\Lambda_{{v}}^M(p,P(v,p))\neq\emptyset$ in some neighbourhood of $(\bar{p},\bar{v})$ under  RCRCQ and the Kuratowski limit conditions.
\begin{proposition}\label{proposition:noneptiness_lagrange}
	Suppose that $\bar{v}\notin C(\bar{p})$. Assume that RCRCQ holds at $(\bar{p},P(\bar{v},\bar{p}))$ (with a neighbourhood $U_0(\bar{p})$) for multifunction $\multifC$, given by \eqref{multifunction_C}, and  $P(\bar{v},\bar{p})\in \liminf\limits_{p\rightarrow \bar{p},\ p\in \setD} \multifC(p)$.
	Let the formula \eqref{eq:representation1} hold, i.e.,
	\begin{equation*}
		\bar{v}-P(\bar{v},\bar{p})=\sum_{i \in I_1\cup \bar{K}} \bar{\lambda}_i \fung_i (\bar{p}),\ \text{where}\  \bar{\lambda}_i>0,\ i\in \bar{K}\subset I_{\bar{p}}(P(\bar{v},\bar{p}))\cap I_2,
	\end{equation*}  and  $g_i(\bar{p})$, $i\in I_1\cup \bar{K}$ are linearly independent. There exist  neighbourhoods $U_1(\bar{p})$, $V_1(0)$ and $M>0$ such that for all $p\in U_1(\bar{p})$ and $v\in \bar{v}+V_1(0)$ we have
	\begin{equation}\label{ineq:bounded_lambdas}
		\Lambda_{v}^M(p,P(v,p))\neq \emptyset.
	\end{equation}
\end{proposition}
\begin{proof}
	Let $\varepsilon>0$. By Proposition \ref{lemma:indices_RCRCQ}, there exist neighbourhoods $U(\bar{p})$, $V(0)$ such that for every $p\in U(\bar{p})$ 
	and any $(x,x^\prime)\in \mbox{gph} N(\cdot;C(p))\cap (\bar{x}+V(0),\bar{v}-\bar{x}+V(0))$, 
	there exists ${L}\subset (I_p(x)\cap I_2)\setminus K \subset (I_{\bar{p}}(\bar{x})\cap I_2)\setminus K$  such that
	the formula \eqref{eq:representation2} holds i.e.,
	\begin{align*}
		& x^\prime=\sum_{i\in I_1} {\lambda}_i \fung_i(p)+\sum_{i\in \bar{K}} {\lambda}_i \fung_i(p)+\sum_{i\in {L}} {\lambda}_i \fung_i(p),\\
		&  \lambda_i> 0,\ i \in \bar{K},\ \lambda_i\geq 0,\ i\in  L,
	\end{align*}	
	where $g_i(p)$ $i\in I_1\cup \bar{K}\cup {L}$ are linearly independent 
	and additionally
	\begin{align}\label{ineq:lambdas2}
		\begin{aligned}
			& \bar{\lambda}_i-\varepsilon\leq {\lambda}_i \leq \bar{\lambda}_i+\varepsilon\quad \forall i \in I_1\cup\bar{K},\\ 
			& 0\leq {\lambda}_i \leq \varepsilon\quad \forall i \in {L}.
		\end{aligned}
	\end{align} 
	Let $V_1(0)$ be such that $V_1(0)\subset \frac{1}{2}V(0)$. By the continuity of  $P(\cdot,\cdot)$ at $(\bar{v},\bar{p})$ (see Theorem \ref{theorem:main_section_RCRCQ} and Proposition \ref{proposition:holderian}), there exist neighbourhoods $U_2(\bar{p})$, $V_2(0)\subset \frac{1}{2}V(0)$ such that
	\begin{equation*}
		P(v,p)\subset P(\bar{v},\bar{p})+V_1(0)
	\end{equation*}
	for all $p\in U_2(\bar{p})$, $v\in \bar{v}+V_2(0)$. Hence,
	\begin{equation*}
		v-P(v,p)\in  \bar{v}+V_2(0) - P(\bar{v},\bar{p})+ V_1(0)\subset \bar{v}-P(\bar{v},\bar{p})+ V(0).
	\end{equation*}
	Let $U_1(\bar{p}):=U(\bar{p})\cap U_2(\bar{p})$. Then for all $p\in U_1(\bar{p})$, $v\in \bar{v}+V(0)$ there exists ${L}\subset I_{\bar{p}}(P(\bar{v},\bar{p}))\cap I_2$ ($L\subset (I_p(v)\cap I_2)\setminus K$) such that
	\begin{align*}
		& P(v,p)-v=\sum_{i\in I_1} {\lambda}_i \fung_i(p)+\sum_{i\in \bar{K}} {\lambda}_i \fung_i(p)+\sum_{i\in {L}} {\lambda}_i \fung_i(p),\\
		&  \lambda_i> 0,\ i \in \bar{K},\ \lambda_i\geq 0,\ i\in  L
	\end{align*}	
	and \eqref{ineq:lambdas2} holds. This means that for  all $p\in U_1(\bar{p})$, $v\in \bar{v}+V(0)$, 
	\begin{equation*}
		\sum_{i\in I_{p}(P(v,p))} |\lambda_i| < \sum_{i\in I_{\bar{p}}(P(\bar{v},\bar{p})) } (|\bar{\lambda}_i|+\varepsilon)=:M,
	\end{equation*}
	i.e. \eqref{ineq:bounded_lambdas} holds.	
\end{proof}

\section{Stable representations}\label{section:stable_representations}
As already noted in Example \ref{example:representation_not_unique},
a number of different index sets $\bar{K}$ could be used in \eqref{eq:representation1}. On the other hand, the set of those index sets $\bar{K}$ for which \eqref{eq:representation1} holds is nonempty (may consists of the empty set only).

\begin{definition}\label{def:equivalent}
	Let $\multifC:\ \setD \rightrightarrows \hilbertH$, be given by \eqref{multifunction_C}. We say that $\multifR:\ \setD \rightrightarrows \hilbertH$ is an \textit{equivalent representation} of $\multifC$, if $R(p)=C(p)$ for all $p\in \setD$  and $R$ is given as
	\begin{equation*}
		R(p):=\left\{ x\in \hilbertH\ \bigg|\ \begin{array}{ll}
			\langle x \ |\ \tilde{g}_i(p) \rangle = \tilde{f}_i(p),& i \in \tilde{I}_1\\
			\langle x \ |\ \tilde{g}_i(p) \rangle \leq  \tilde{f}_i(p),& i \in \tilde{I}_2
		\end{array}\right\},
	\end{equation*}
	where $\tilde{f}_i:\ \setD \rightarrow \mathbb{R}$, $\tilde{g}_i:\ \setD \rightarrow \hilbertH$, $i\in \tilde{I}_1\cup \tilde{I}_2$ are locally Lipschitz on $\setD$ and $\tilde{I}_1\cup \tilde{I}_2$ is a finite, nonempty set. For a given representation $\multifR$ of $\multifC$, we define $I_{{p}}^\multifR(x)=\{ i\in \tilde{I}_1\cup \tilde{I}_2 \mid \langle x \mid \tilde{g}_i(p) \rangle = \tilde{f}_i(p)  \}$, $p\in \setD$, $x\in \hilbertH$.
\end{definition}
Consider any equivalent representation $\multifR$ of $\multifC$. Let $\bar{p}\in \setD$, $\bar{x}\in R(\bar{p})=C(\bar{p})$, $I_{\bar{p}}^\multifR(\bar{x})\neq \emptyset$ and let the 
following formula (c.f., formula \eqref{eq:representation1})	holds		 
\begin{equation}\label{representation:corollary}
	\bar{v}-\bar{x}=\sum_{i \in \tilde{I}_1\cup \bar{K}} \bar{\lambda}_i \tilde{g}_i (\bar{p}),\  \bar{\lambda}_i>0,\ i\in \bar{K}\subset I_{\bar{p}}^\multifR(\bar{x})\cap \tilde{I}_2,
\end{equation}  where  $\tilde{g}_i(\bar{p})$, $i\in \tilde{I}_1\cup \bar{K}$ are linearly independent.

For a given representation $\multifR$, for any index set $L$\footnote{We also admit the case $L=\emptyset$.} satisfying
\begin{align}\label{conditions_setL}
	\begin{aligned}
		&L\subset (I_{\bar{p}}^\multifR(\bar{x})\cap \tilde{I}_2)\setminus\bar{K}\\
		&\tilde{g}_i(\bar{p}),\ i \in \tilde{I}_1\cup\bar{K}\cup L,\ \text{linearly independent}
	\end{aligned}
\end{align}
we define multifunction $\multifR_{\bar{K},L}:\ \setD \rightrightarrows \hilbertH$ as $\multifR_{\bar{K},L}(p):=R_{\bar{K},L}(p)$,
\begin{equation*}
	R_{\bar{K},L}(p):=\left\{ x\in \hilbertH\ \bigg|\ \begin{array}{ll}
		\langle x \ |\ \tilde{g}_i(p) \rangle = \tilde{f}_i(p),& i \in \tilde{I}_1\cup \bar{K}\cup L,\\
		\langle x \ |\ \tilde{g}_i(p) \rangle \leq  \tilde{f}_i(p),& i \in \tilde{I}_2\setminus (\bar{K}\cup L)
	\end{array}\right\}.
\end{equation*}
Note that $\bar{x}\in R_{\bar{K},L}(\bar{p})$ for any index set L satisfying \eqref{conditions_setL} and,  in general, $R_{\bar{K},L}(p)\neq C_{\bar{K},L}(p)$, $p\in \setD$.

\begin{definition}\label{def:stability}
	Let $\bar{p}\in \setD$, $\bar{v}\notin C(\bar{p})$  and $\bar{x}\in C(\bar{p})$. We say that multifunction $\multifC:\ \setD \rightrightarrows \hilbertH$, given by \eqref{multifunction_C}, has a \textit{stable representation} (in the sense of Kuratowski limit) at $(\bar{p},\bar{v},\bar{x})$ 
	if there exists an equivalent representation $\multifR$ of $\multifC$ for which there exists $\bar{K}\subset I_{\bar{p}}^\multifR(\bar{x})\cap \tilde{I}_2$ such that  \eqref{representation:corollary} holds and
	\begin{equation}\label{assumption:Kuratowski}
		\bar{x}\in \liminf\limits_{p\rightarrow \bar{p},\ p\in \setD}R_{\bar{K},L}(p) \quad \text{for any $L$ satisfying \eqref{conditions_setL}.}
	\end{equation}	
	We say that   $\multifR$ is a \textit{stable  representation} of $\multifC$  at $(\bar{p},\bar{v},\bar{x})$ if there exists $\bar{K}$ such that \eqref{representation:corollary} and \eqref{assumption:Kuratowski} hold.
\end{definition}
Let us note that if multifunction $\multifC:\ \setD \rightrightarrows \hilbertH$, given by \eqref{multifunction_C}, has a stable representation $\multifR$ at $(\bar{p},\bar{v},\bar{x})$ then $\bar{x}\in \liminf\limits_{p\rightarrow \bar{p},\ p\in \setD} C(p)=\liminf\limits_{p\rightarrow \bar{p},\ p\in \setD} R(p)$.

By Proposition \ref{lemma:indices_RCRCQ} and Theorem \ref{theorem:main_section_RCRCQ}, we obtain the following corollary.
\begin{corollary}\label{remark:RCRCQ_extended}
	Let $\bar{p}\in \setD$, $\bar{x}\in C(\bar{p})$, $I_{\bar{p}}(\bar{x})\neq \emptyset$ and $\bar{v}-\bar{x}\in N(\bar{x};C(\bar{p}))$. 
	Assume that there exists an equivalent representation $\multifR$ of $\multifC$ satisfying
	\begin{enumerate}[-]
		\item  RCRCQ holds for multifunction $\multifR$ at $(\bar{p},\bar{x})$,
		\item $\multifR$ is a stable representation  at $(\bar{p},\bar{v},\bar{x})$ with some set $\bar{K}\subset I_{\bar{p}}^\multifR(\bar{x})\cap \tilde{I}_2$. 
	\end{enumerate}
	There exists a constant $\ell>0$ such that for all $\varepsilon>0$ one can find neighbourhoods $U(\bar{p})$ and $V(0)$ satisfying
	
	\begin{enumerate}[(1)]
		\item\label{cond:1_corollary}  for any $p\in U(\bar{p})$ 
		and any $(x,x^\prime)\in \mbox{gph} N(\cdot;C(p))\cap (\bar{x}+V(0),\bar{v}-\bar{x}+V(0))$, it exists $\hat{L}\subset (I_{{p}}^\multifR({x})\cap \tilde{I}_2)\setminus \bar{K} \subset (I_{\bar{p}}^\multifR(\bar{x})\cap \tilde{I}_2)\setminus \bar{K}$ satisfying \eqref{conditions_setL} such that
		\begin{align*}
			& \exists\, {\lambda}_i,\ i \in \tilde{I}_1\cup\bar{K}\cup {\hat{L}}, \quad x^\prime = \sum_{i \in \tilde{I}_1\cup \bar{K}\cup {\hat{L}}} \lambda_i \tilde{g}_i(p),\\
			& \bar{\lambda}_i-\varepsilon\leq {\lambda}_i \leq \bar{\lambda}_i+\varepsilon\quad \forall i \in \tilde{I}_1\cup \bar{K},\\ 
			&  0< {\lambda}_i \leq \varepsilon\quad \forall i \in {\hat{L}},
		\end{align*}
		\item\label{cond:2_corollary}  for every $L$ satisfying \eqref{conditions_setL}, every $p_1,p_2\in U(\bar{p})$ and  every  $x_1\in (\bar{x}+V(0))\cap R_{\bar{K},L}(p_1)$ there exists $x_2\in R_{\bar{K},{L}}(p_2)$
		such that
		\begin{align*}
			&\|x_1-x_2\|\leq \ell\|p_1-p_2\|,
		\end{align*}
		i.e., the set-valued mapping $\multifR_{\bar{K},L}$ is Lipschitz-like at $(\bar{p},\bar{x})$.
	\end{enumerate}
	
\end{corollary}
\begin{proof}
	Clearly, RCRCQ holds for any $\multifR_{\bar{K},L}$ at $(\bar{p},\bar{x})$, with $L$ satisfying \eqref{conditions_setL}. 
	
	By Proposition \ref{lemma:indices_RCRCQ} applied to $\multifR$, there exist neighbourhoods $U_1(\bar{p})$, $V_1(0)$ such that assertion \eqref{cond:1_corollary}  holds. 
	
	By Theorem \ref{theorem:main_section_RCRCQ} applied to $\multifR$ at $(\bar{p},\bar{x})$, for any $L$ satisfying \eqref{conditions_setL}, the multifunction $\multifR_{\bar{K},L}$ is Lipschitz-like at $(\bar{p},\bar{x})$ with neighbourhoods $U_L(\bar{p})$, $V_L(0)$ and constant $\ell_L>0$, i.e., assertion \eqref{cond:2_corollary} holds. 
	
	The existence of neighbourhoods $U(\bar{p})$, $V(0)$ and constant $\ell>0$ satisfying the assertion follows from the fact that there is a finite number of sets $L$ satisfying \eqref{conditions_setL}.

\end{proof}

In Theorem \ref{theorem:necessary} below we use the following assumption (H1).
\begin{enumerate}
	\item[(H1)] There 	exist an equivalent representation $\multifR$ of $\multifC$, given by \eqref{multifunction_C}, with 
	$$\bar{v}-P(\bar{v},\bar{p})=\sum_{i\in \tilde{I}_1\cup \bar{K}} \bar{\lambda}_i \tilde{g}_i(\bar{p}) $$
	where $\bar{\lambda}_i>0$, $i\in \bar{K}$,  $\tilde{g}_i(\bar{p})$, $i\in \tilde{I}_1\cup \bar{K}$ are linearly independent ($\bar{K}\subset I_2\cap I_{\bar{p}}^\multifR(P(\bar{v},\bar{p}))$),
	and neighbourhoods $U(\bar{p})$, $W(\bar{v})$ such that 
	\begin{enumerate}
		\item $\bar{K}\subset I_{p}^\multifR(P(v,p))$ for all $p\in U(\bar{p})$,  $v\in W(\bar{v})$, 
		\item for any $p_n\rightarrow \bar{p}$ and any $L\subset (\tilde{I}_2\cap I_{\bar{p}}^\multifR(P(\bar{v},\bar{p})))\setminus \bar{K}$ such that $\tilde{g}_i(\bar{p})$, $i\in \tilde{I}_1\cup \bar{K}\cup L$ are linearly independent   there exist sequence $v_n\rightarrow \bar{v}$, such that $ L \subset I_{p_n}^\multifR(P(v_n,p_n))$ for $n$ sufficiently large.
	\end{enumerate}	
\end{enumerate}

Below we show that if an equivalent representation $\multifR$ of $\multifC$ satisfies assumption (H1), then the stability of  $\multifR$ (in the sense of Definition \ref{def:stability}) is necessary for continuity of projection operator $P$.

\begin{theorem}\label{theorem:necessary}
	Let $\multifC:\ \setD \rightrightarrows \hilbertH$ be given by \eqref{multifunction_C}. Suppose that (H1)	 holds, i.e., there exists an equivalent representation $\multifR$ of $\multifC$ satisfying conditions (a) and (b). If projection $P:\ \hilbertG \times \setD\rightarrow \hilbertH$, with $P(\cdot,\cdot)$  given by \eqref{projection}, is continuous at $(\bar{v},\bar{p})\in \hilbertH \times \setD$, 
	then the representation $\multifR$ of $\multifC$ is stable at $(\bar{p},\bar{v},P(\bar{v},\bar{p}))$.
\end{theorem}

\begin{proof}
	
	By contradiction suppose, that 
	representation $\multifR$ of $\multifC$ is not stable at
	$(\bar{p},\bar{v},P(\bar{v},\bar{p}))$, i.e, 
	for any $\tilde{K}$ such that
	\begin{equation}\label{eq:representation_conditions}
		\bar{v}-P(\bar{v},\bar{p})=\sum_{i \in \tilde{I}_1\cup \tilde{K}} \bar{\lambda}_i \tilde{g}_i (\bar{p}),\ \text{where}\ \tilde{\lambda}_i>0,\ i\in \tilde{K}\subset I_{\bar{p}}^\multifR(P(\bar{v},\bar{p}))\cap \tilde{I}_2
	\end{equation}  and  $\tilde{g}_i(\bar{p})$, $i\in 	\tilde{I}_1\cup \tilde{K}$ are linearly independent, there exists $\tilde{L}$ satisfying \eqref{conditions_setL} such that $P(\bar{v},\bar{p})\notin \liminf\limits_{p\rightarrow \bar{p},\ p\in \setD} R_{\tilde{K},\tilde{L}}(p)$. In particular, \eqref{eq:representation_conditions} holds for $\tilde{K}=\bar{K}$ and for any $\tilde{L}\subset (I_2\cap I_{\bar{p}}^\multifR(P(\bar{v},\bar{p})))\setminus \bar{K}$ such that $g_i(\bar{p})$, $i\in \tilde{I}_1\cup \bar{K}\cup \tilde{L}$ are linearly independent.

	By assumption that $P(\bar{v},\bar{p})\notin \liminf\limits_{p\rightarrow \bar{p},\ p\in \setD} R_{\bar{K},\tilde{L}}(p)$, there exists a neighbourhood $V(0)$ such that in every neighbourhood of $\bar{p}$ one can find element $p$ such that  $(P(\bar{v},\bar{p})+V(0))\cap R_{\bar{K},\tilde{L}}(p)=\emptyset$,
	i.e., there exists a sequence $p_n\rightarrow\bar{p}$ such that 
	$(P(\bar{v},\bar{p})+V(0))\cap R_{\bar{K},\tilde{L}}(p_n)=\emptyset$.
	
	Consider a sequence  $v_n\rightarrow \bar{v}$ satisfying condition (b) of (H1). 
	Then
	\begin{equation*}
		v_n-P(v_n,p_n)=\sum_{i \in  \tilde{I}_1\cup \bar{K}\cup \tilde{L}} \lambda_i^n \tilde{g}_i(p_n),\quad \lambda_i^n\geq 0, \ i\in \bar{K}\cup \tilde{L}.
	\end{equation*}
	This formula implies that $ \tilde{I}_1\cup \bar{K}\cup \tilde{L}\subset I_{p_n}^\multifR(P(v_n,p_n))$, $P(v_n,p_n)\in R_{\bar{K},\tilde{L}}(p_n)$ and $v_n-P(v_n,p_n)\in N(x_n,R_{\tilde{L}}(p_n))$. Thus $P(v_n,p_n)=P_{R_{\bar{K},\tilde{L}}(p_n)}(v_n)$. Hence, $P(v_n,p_n)\notin P(\bar{v},\bar{p})+ V(0)$, which means that $P(\cdot,\cdot)$ is not continuous at $(\bar{v},\bar{p})$.		
	
\end{proof}

\section{Main results}
\label{section:main_results}
In this section
we prove local Lipschitzness of projections onto moving closed convex sets $C(p)$ defined by \eqref{multifunction_C}.
In view of Theorem \ref{theorem:main_section_RCRCQ} in order to apply Theorem \ref{theorem:Mordukhovic_2} we need to investigate Lipschitz-likeness of the 	graphical subdifferential  mapping $Gr$ given by \eqref{mulfiunction:K}.

We start with the following technical fact.

\begin{proposition}\label{prop:eqivalent} Let $(\bar{p},\bar{x},\bar{v}-\bar{x})$ be such that $\bar{x}\in C(\bar{p})$, $\bar{v}-\bar{x}\in N(\bar{x};C(\bar{p}))$, and $I_{\bar{p}}(\bar{x})=\{i \in I_1\cup I_2\ |\ \langle \bar{x} \ |\ \fung_i(\bar{p})\rangle =\funf_i(\bar{p})  \}\neq \emptyset$.
	The following conditions are equivalent:
	\begin{enumerate}[(i)]
		\item \label{cond:1} The graphical subdifferential mapping $Gr:\ \setD\rightrightarrows\hilbertH\times \hilbertH$ defined as
		\begin{equation*}
			Gr(p)=\{ (x,x^\prime)\ |\ x\in C(p),\ x^\prime \in N(x;C(p))  \}
		\end{equation*}
		is Lipschitz-like around $(\bar{p},\bar{x},\bar{v}-\bar{x})$
		\item \label{cond:2} There exist $\ell>0$ and neighbourhoods $U(\bar{p})$, $V(0)$ in $\hilbertH$ such that
		\begin{align}
			& \forall\ p_1,p_2\in U(\bar{p})\notag\\
			& \forall\ x_1\in C(p_1)\cap (\bar{x}+V(0)),\ x_1^\prime \in N(x_1;C(p_1))\cap (\bar{v}-\bar{x}+V(0))\notag\\ 
			&\exists\ x_2\in C(p_2),\ x_2^\prime \in N(x_2;C(p_2))\ \text{satisfying}\notag\\
			& \|x_1-x_2\|\leq \ell \|p_1-p_2\|, \tag{a}\label{ineq:2.a}\\
			& \|x_1^\prime-x_2^\prime\|\leq \ell \|p_1-p_2\|.\tag{b}\label{ineq:2.b}
		\end{align}
	\end{enumerate}
\end{proposition}
\begin{proof}
	By \eqref{cond:1}, there exist neighbourhoods $U(\bar{p})$, $V(0)$ such that for every $(p_1,p_2)\in U(\bar{p})$
	\begin{align*}
		\begin{aligned}
			\gph N(\cdot;C(p_1))\cap &(\bar{x}+V(0),\bar{v}-\bar{x}+V(0))\\
			&\subset \gph N(\cdot;C(p_2))+\ell \|p_1-p_2\|B(0,1),
		\end{aligned}
	\end{align*}
	i.e., for all $(x_1,x_1^\prime) \in \gph N(\cdot;C(p_1))\cap (\bar{x}+V(0),\bar{v}-\bar{x}+V(0))$ there exists $(x_2,x_2^\prime)\in \gph N(\cdot;C(p_2)) $  such that
	\begin{equation*}
		(x_1,x_1^\prime) \in (x_2,x_2^\prime)+\ell \|p_1-p_2\|B(0,1),
	\end{equation*}
	where $B(0,1)\subset \hilbertH\times \hilbertH$   is the open unit ball in
	$\hilbertH\times \hilbertH$. Hence, 
	\begin{equation}\label{inc:3}
		\|x_1-x_2\|+\|x_1^\prime-x_2^\prime\|\leq \ell \|p_1-p_2\|,
	\end{equation}
	where $x_{1}\in C(p_{1})$ and $x_{2}\in C(p_{2})$, $x_{1}'\in N(x_{1};C(p_{1}))$, $x_{2}'\in N(x_{2};C(p_{2}))$,
	which implies \eqref{ineq:2.a} and \eqref{ineq:2.b}.
	The converse implication is immediate.
\end{proof}
\begin{remark}
	Let us note that in \eqref{inc:3} we use the norm $\|\cdot\|_1$ in the Cartesian product $\hilbertH\times \hilbertH$. Clearly, any other equivalent norm can be used at this point.
\end{remark}

Let $\bar{p}\in \setD$, $\bar{v}\notin C(\bar{p})$ and $\bar{x}=P(\bar{v},\bar{p})$.  
From \cite[Theorem 6.41]{deutsch2001best} (see also \cite{closedform}) the following representation holds
\begin{equation*}
	\bar{v}-\bar{x}=\sum_{i \in {I_{\bar{p}}}(\bar{x})}\hat{\lambda}_i \fung_i(\bar{p})\quad \text{with}\  \hat{\lambda}_i\geq 0\ \text{for}\ i \in I_{\bar{p}}(\bar{x})\cap I_2.
\end{equation*}

In the 	proposition below we give sufficient conditions for the graphical subdifferential mapping Gr given by \eqref{mulfiunction:K} to be Lipschitz-like at $(\bar{p},\bar{x},\bar{v}-\bar{x})$.

\begin{proposition}\label{prop:2ab2}
	Let $\bar{p}\in \setD$,  $\bar{v}\notin C(\bar{p})$.
	Assume that there exists an equivalent stable representation $\multifR$ of $\multifC$ at $(\bar{p},\bar{v},P(\bar{v},\bar{p}))$, given by \eqref{multifunction_C}, (with  set $\bar{K}$)  and RCRCQ holds for $\multifR$  at $(\bar{p},P(\bar{v},\bar{p}))$
	Then the graphical subdifferential mapping Gr, given by \eqref{mulfiunction:K}, is Lipschitz-like at $(\bar{p},P(\bar{v},\bar{p}),\bar{v}-P(\bar{v},\bar{p}))$.
\end{proposition}
\begin{proof}
	
	Let $\varepsilon>0$.  Let $U(\bar{p})$, $V(0)$ be as in Corollary \ref{remark:RCRCQ_extended}.
	We have
	\begin{equation*}
		\bar{v}-P(\bar{v},\bar{p})=\sum_{i \in \tilde{I}_1\cup \bar{K}} \bar{\lambda}_i \tilde{g}_i (\bar{p}),\  \bar{\lambda}_i>0,\ i\in \bar{K}\subset I_{\bar{p}}^\multifR(\bar{x})\cap \tilde{I}_2,
	\end{equation*}  where  $\tilde{g}_i(\bar{p})$, $i\in \tilde{I}_1\cup \bar{K}$ are linearly independent.
	
	Now, let $p_1\in U(\bar{p})$ and $x_1\in (P(\bar{v},\bar{p})+V(0))\cap C(p_1)$, $x_1^\prime \in N(x_1;C(p_1))\cap (\bar{v}-P(\bar{v},\bar{p})+V(0))$. By 
	Corollary \ref{remark:RCRCQ_extended}, there exists $\tilde{L}\subset I_{\bar{p}}^\multifR(P(\bar{v},\bar{p}))\cap \tilde{I}_2\setminus \bar{K}$ such that $x_1\in R_{\bar{K},\tilde{L}}(p_1)$ and 
	\begin{align*}
		& x_1^\prime = \sum_{i\in \tilde{I}_1\cup \bar{K}  } \lambda_i^1 g_i(p_1)+\sum_{i\in \tilde{L}} \lambda_i^1 \tilde{g}_i(p_1),\\
		& \text{where}\ \bar{\lambda}_i-\varepsilon\leq \lambda_i^1 \leq \bar{\lambda}_i+\varepsilon,\quad i \in \tilde{I}_1\cup \bar{K},\\
		& 0< \lambda_i^1\leq \varepsilon,\quad i \in \tilde{L},
	\end{align*} 
	and for any $p_2\in U(\bar{p})$  there exists $x_2\in  R_{\bar{K},\tilde{L}}(p_2)\subset C(p_2)$ such that 
	\begin{equation*}
		\|x_1-x_2\|\leq \ell^1 \|p_1-p_2\|
	\end{equation*}
	for some $\ell^1>0$.
	Since $x_2\in R_{\bar{K},\tilde{L}}(p_2)$ we have
	\begin{equation*}
		x_2^\prime:= \sum_{i \in \tilde{I}_1\cup \bar{K}\cup \tilde{L}} \lambda_i^1 \tilde{g}_i(p_2)\in N(x_2;R_{\bar{K},\tilde{L}}(p_2)).
	\end{equation*}
	Then
	\begin{align*}
		&\|x_1^\prime-x_2^\prime\|=\|\sum_{i \in \tilde{I}_1\cup \bar{K}\cup \tilde{L}} \lambda_i^{1}\tilde{g}_i(p_1)-\sum_{i \in  \tilde{I}_1\cup \bar{K}\cup \tilde{L}} \lambda_i^{1} \tilde{g}_i(p_2)\|\\
		&\leq \sum_{i \in  \tilde{I}_1\cup \bar{K}\cup \tilde{L}} |\lambda_i^{1}|\|\tilde{g}_i(p_1)-\tilde{g}_i(p_2)\|\leq \sum_{i \in  \tilde{I}_1\cup \bar{K}\cup \tilde{L}} |\lambda_i^{1}| \ell_{\tilde{g}_i} \|p_1-p_2\|\\
		&\leq \sum_{i \in  \tilde{I}_1\cup \bar{K}\cup \tilde{L}} (|\bar{\lambda}_i|+\varepsilon) \ell_{\tilde{g}_i} \|p_1-p_2\|
		\leq \ell^2  \|p_1-p_2\|,
	\end{align*}	
	where we put $\ell^2:=\sum_{i \in I_{\bar{p}}^\multifR(P(\bar{v},\bar{p}))} (|\bar{\lambda}_i|+\varepsilon) \ell_{\tilde{g}_i} $.
\end{proof}
Now we are ready to establish our main theorem.
\begin{theorem}\label{theorem:main}
	Let $\hilbertH$ be a Hilbert space and let $\setD\subset \hilbertG$ be a nonempty set of a normed space $\hilbertG$ and $\bar{p}\in \setD$. Let $\multifC:\ \setD \rightrightarrows \hilbertH$ be as in \eqref{multifunction_C}, where  $f_i:\ \setD\rightarrow \mathbb{R}$, $g_i:\ \setD\rightarrow\hilbertH$, $i\in I_1\cup I_2\neq \emptyset,$ $I_1=\emptyset \vee \{1,\dots,q\}$, $I_2=\emptyset \vee \{q+1,\dots,m\}$  are   locally Lipschitz  on $\setD$. 
	Let $\bar{p}\in \setD$, $\bar{v}\notin C(\bar{p})$.
	Assume that  there exists an equivalent representation $\multifR$ of $\multifC$ such that
	\begin{enumerate}[-]
		\item  RCRCQ holds for multifunction $\multifR$ at $(\bar{p},\bar{x})$,
		\item $\multifR$ is a stable representation of $\multifC$ at $(\bar{p},\bar{v},\bar{x})$ with some set $\bar{K}\subset I_{\bar{p}}^\multifR(\bar{x})\cap \tilde{I}_2$. 
	\end{enumerate}
	There exist neighbourhoods $W(\bar{v})$, $U(\bar{p})$ such that the Lipschitzian estimate 
	\begin{equation}\label{lipschitzian_stability_of_projections}
		\| (v_1-v_2)-2[P(v_1,p_1)-P(v_2,p_2)]\|\leq \|v_1-v_2\|+\ell^0\|p_1-p_2\|  
	\end{equation}
	holds for all $(v_1,p_1),(v_2,p_2)\in W(\bar{v})\times U(\bar{p})$ with some positive constant $\ell^0$.
\end{theorem}

In particular, we get the following result. 
\begin{theorem}\label{theorem:main_2}
	Let $\hilbertH$ be a Hilbert space and let $\setD\subset \hilbertG$ be a nonempty set of a normed space $\hilbertG$ and $\bar{p}\in \setD$. Let $\multifC:\ \setD \rightrightarrows \hilbertH$ be as in \eqref{multifunction_C}, where  $f_i:\ \setD\rightarrow \mathbb{R}$, $g_i:\ \setD\rightarrow\hilbertH$, $i\in I_1\cup I_2\neq \emptyset,$ $I_1=\emptyset \vee \{1,\dots,q\}$, $I_2=\emptyset \vee \{q+1,\dots,m\}$  are   locally Lipschitz  on $\setD$. 
	Let $\bar{p}\in \setD$, $\bar{v}\notin C(\bar{p})$.
	Assume that  
	\begin{enumerate}[(1)]
		\item\label{cond:1_main}  RCRCQ holds for multifunction $\multifC$ at $(\bar{p},\bar{x})$,
		\item\label{cond:2_main} $\multifC
		$ is a stable representation  at $(\bar{p},\bar{v},\bar{x})$ with some set $\bar{K}\subset I_{\bar{p}}(\bar{x})\cap \tilde{I}_2$. 
	\end{enumerate}
	There exist neighbourhoods $W(\bar{v})$, $U(\bar{p})$ such that the Lipschitzian estimate 
	\begin{equation}\label{lipschitzian_stability_of_projections_2}
		\| (v_1-v_2)-2[P(v_1,p_1)-P(v_2,p_2)]\|\leq \|v_1-v_2\|+\ell^0\|p_1-p_2\|  
	\end{equation}
	holds for all $(v_1,p_1),(v_2,p_2)\in W(\bar{v})\times U(\bar{p})$ with some positive constant 
	$\ell^0$.
\end{theorem}
\begin{proof}\textit{of Theorem \ref{theorem:main}}. 
	The proof follows directly from Theorem \ref{theorem:Mordukhovic_2}, Theorem \ref{theorem:main_section_RCRCQ_2} and Proposition \ref{prop:2ab2}.
\end{proof}
Clearly, by \eqref{lipschitzian_stability_of_projections},
\begin{equation*}
	\|P(v_1,p_1)-P(v_2,p_2)\|\leq \|v_1-v_2\|+\frac{\ell^0}{2}\|p_1-p_2\|.
\end{equation*}
If the multifunction $\multifC$ is constant around $\bar{p}$, then assumptions 
of Theorem \ref{theorem:main} are satisfied.
\begin{corollary}
	Under assumptions of Theorem \ref{theorem:main}, projection  of a given fixed $\bar{v}$ onto $C(p)$, $p\in \setD$, i.e.,  
	\begin{equation*}
		P_{\bar{v}}(p):=P(\bar{v},p),\quad p\in \setD
	\end{equation*}
	is locally Lipschitz at $\bar{p}$.	
\end{corollary}

Example \ref{example:C_L_not_sufficient} shows that for a given representation of  $\multifC$
one can not expect that there exists $\bar{K}$ such that \eqref{representation:corollary} and \eqref{conditions_setL} holds.
\begin{example}\label{example:C_L_not_sufficient}
	Let $\hilbertH=\mathbb{R}^2$, $\hilbertG=\mathbb{R}$, $\bar{p}=0$, $\bar{v}=(1,1)$ and
	\begin{equation}\label{multifunction:example_not_RCRCQ}
		\multifC(p)=\left\{ x\in \mathbb{R}^2 \ \bigg|\  \begin{array}{l}
			\langle x \ |\ (1,0) \rangle \leq 0, \\
			\langle x \ |\ (0,1) \rangle \leq 0, \\
			\langle x \ |\ (1,1) \rangle \leq p \\
		\end{array}  \right\}.
	\end{equation}
	In this case we have
	\begin{equation*}
		P(\bar{v},p)=\left\{\begin{array}{lll}
			(0,0) & \text{if} & p\geq 0,\\
			(\frac{p}{2},\frac{p}{2}) & \text{if} & p< 0
		\end{array}\right. 
	\end{equation*}
	and $P(\bar{v},p)$ is a Lipschitz function of $p$. On the other hand,
	\begin{equation*}
		\bar{v}-P(\bar{v},\bar{p})=\left\{\begin{array}{lll}
			1\cdot (1,0)+1\cdot (0,1) & \implies & \bar{K}=\{1,2\} \text{ and }C_\emptyset=\emptyset  \text{ for }  p<0,\\
			1\cdot (1,1) & \implies & \bar{K}=\{3\} \text{ and }C_\emptyset=\emptyset\text{ for }  p>0
		\end{array}\right. .			
	\end{equation*}
	Hence, the representation \eqref{multifunction:example_not_RCRCQ} of $\multifC$ is not stable  at $(0,(0,0),(1,1))$.
\end{example}
\begin{remark}
	Let us note that 
	multifunction $\multifC$, given by \eqref{multifunction:example_not_RCRCQ}, can be equivalently represented as 
	\begin{equation}\label{multifunction:example_stable}
		{\multifC}(p)=\left\{ x\in \mathbb{R}^2 \ \bigg|\  \begin{array}{l}
			\langle x \ |\ (1,0) \rangle \leq 0, \\
			\langle x \ |\ (0,1) \rangle \leq 0, \\
			\langle x \ |\ (1,1) \leq g_3(p) \\
		\end{array}  \right\},
	\end{equation}
	where
	\begin{equation*}
		g_3(p)=\left\{\begin{array}{lll}
			p & \text{if} & p< 0\\
			0 & \text{if} & p\geq 0\\
		\end{array}
		\right. .
	\end{equation*}
	The representation given in \eqref{multifunction:example_stable} is stable at $(\bar{p},\bar{v},P(\bar{v},\bar{p}))$ and RCRCQ holds at $(\bar{p},P(\bar{p},\bar{v}))$.
\end{remark}

\begin{corollary}
	Suppose that in the definition of the set $C(p)$, $p\in \setD$, given in \eqref{multifunction_C}, $I_2=\emptyset$, i.e.,  
	\begin{equation*}
		C(p):=\left\{ x\in \hilbertH\ \bigg|\ 
		\begin{array}{ll}
			\langle x \ |\ g_i(p)\rangle = f_i(p), &i \in I_1
		\end{array}
		\right\}.
	\end{equation*}
	Let $\bar{p}\in \setD$, $\bar{v}\notin C(\bar{p})$, $\bar{x}=P(\bar{v},\bar{p})$ and  the following hold:
	\begin{enumerate}[(1)]
		\item RCRCQ holds for multifunction $\multifC$ at $(\bar{p},\bar{x})$, i.e., there exists a neighbourhood $U(\bar{p})$ such that
		\begin{equation*}
			\rank \{g_i(p),\ i\in I_1   \}=\rank \{ g_i(\bar{p}),i\in I_1  \},\quad p\in U(\bar{p}).
		\end{equation*}
		\item $\bar{x}\in \liminf\limits_{p\rightarrow \bar{p},\ p\in \setD} \multifC(p)$.
	\end{enumerate}
	Then the projection $P(v,p)$ is locally Lipschitz at $(\bar{v},\bar{p})$.
\end{corollary}
When LICQ condition holds for set $C(\bar{p})$ at $P(\bar{v},\bar{p})$, i.e., when $g_i(\bar{p})$, $i\in I_{\bar{p}}(P(\bar{v},\bar{p}))$ are linearly independent, Theorem \ref{theorem:main_2} can rewritten in  a  considerably simplified form.

\begin{theorem}\label{theorem:main_3}	Let $\hilbertH$ be a Hilbert space and let $\setD\subset \hilbertG$ be a nonempty set of a normed space $\hilbertG$ and $\bar{p}\in \setD$. Let $\multifC:\ \setD \rightrightarrows \hilbertH$ be as in \eqref{multifunction_C}, where  $f_i:\ \setD\rightarrow \mathbb{R}$, $g_i:\ \setD\rightarrow\hilbertH$, $i\in I_1\cup I_2\neq \emptyset$, 
	$I_1=\emptyset \vee \{1,\dots,q\}
	$, $I_2=
	\{q+1
	,\dots,m\}\vee \emptyset$  are   locally Lipschitz  on $\setD$.
	Let  $\bar{v}\notin C(\bar{p})$ and LICQ hold for set $C(\bar{p})$ at $P(\bar{v},\bar{p})$. There exist  neighbourhoods $W(\bar{v})$,  $U(\bar{p})$  such that the Lipschitzian estimate 
	\begin{equation*}
		\| (v_1-v_2)-2[P(v_1,p_1)-P(v_2,p_2)]\|\leq \|v_1-v_2\|+\ell^0\|p_1-p_2\|  
	\end{equation*}
	holds for all $(v_1,p_1),(v_2,p_2)\in W(\bar{v})\times U(\bar{p})$ with some positive constant 
	$\ell^0$.
\end{theorem}
\begin{proof}
	We have
	\begin{equation*}
		\bar{v}-P(\bar{v},\bar{p})=\sum_{i \in {I}_1\cup \bar{K}} \bar{\lambda}_i {g}_i (\bar{p}),\  \bar{\lambda}_i>0,\ i\in \bar{K}\subset I_{\bar{p}}(\bar{x})\cap {I}_2,
	\end{equation*}
	where  $g_i(\bar{p})$, $\bar{p}\in \setD$, $i\in I_1\cup K\subset I_{\bar{p}}(P(\bar{v},\bar{p}))$ are linearly independent.
	Thus \eqref{cond:1_main} of Theorem \ref{theorem:main_2} is satisfied. 
	
	Now we show \eqref{cond:2_main} of Theorem \ref{theorem:main_2}. Observe that LICQ hold for set 	$C_{\bar{K},L}(\bar{p})$ at $P(\bar{v},\bar{p})$   with any $L$ satisfying \eqref{conditions_setL}. Hence, MFCQ holds for set $C_{\bar{K},L}(\bar{p})$ at $P(\bar{v},\bar{p})$  with any $L$ satisfying \eqref{conditions_setL}. Thus, by Theorem 2.87 of \cite{Bonnans_Shapiro}, for any $L$ satisfying \eqref{conditions_setL}, there exist $\alpha>0$ and a neighbourhood $U(\bar{p})$ of $\bar{p}$ such that for all $p\in U(\bar{p})$ 
	\begin{equation*}
		\dist (\bar{x},C_{\bar{K},L}(p))\leq \alpha \left(\sum_{i\in I_1\cup \bar{K}\cup L} | \langle \bar{x}  \mid g_i(p) \rangle - f_i(p)| + \sum_{i\in I_{\bar{p}(\bar{x})}\setminus (\bar{K}\cup L)}  [ \langle \bar{x}  \mid g_i(p) \rangle - f_i(p)]_+  \right).
	\end{equation*}
	This implies  that 
	$\bar{x}\in \liminf\limits_{p\rightarrow \bar{p},\ p\in \setD} \multifC_{\bar{K},L}(p)$ for any $L$ satisfying \eqref{conditions_setL}, i.e. assumption \eqref{cond:2_main} of Theorem \ref{theorem:main_2} is satisfied, which proves the assertion. 
\end{proof}
In view of proof of Theorem \ref{theorem:main_3} the following corollary holds. 
\begin{corollary}
	Let $\hilbertH$ be a Hilbert space and let $\setD\subset \hilbertG$ be a nonempty set of a normed space $\hilbertG$ and $\bar{p}\in \setD$. Let $\multifC:\ \setD \rightrightarrows \hilbertH$ be as in \eqref{multifunction_C}, where  $f_i:\ \setD\rightarrow \mathbb{R}$, $g_i:\ \setD\rightarrow\hilbertH$, $i\in I_1\cup I_2\neq \emptyset,$ $I_1=\emptyset \vee \{1,\dots,q\}$, $I_2=\emptyset \vee \{q+1,\dots,m\}$  are   locally Lipschitz  on $\setD$. 
	Let $\bar{v}\notin C(\bar{p})$ 	and
	\begin{equation*}
		\bar{v}-P(\bar{v},\bar{p})=\sum_{i \in {I}_1\cup \bar{K}} \bar{\lambda}_i {g}_i (\bar{p}),\  \bar{\lambda}_i>0,\ i\in \bar{K}\subset I_{\bar{p}}(\bar{x})\cap {I}_2,
	\end{equation*}
	where $g_i(\bar{p})$, $i\in I_1\cup \bar{K}$ are linearly independent.
	Assume that  
	\begin{enumerate}[(1)]
		\item  RCRCQ holds for multifunction $\multifC$ at $(\bar{p},\bar{x})$,
		\item  MFCQ holds for set $C_{\bar{K},L}(\bar{p})$ at $P(\bar{v},\bar{p})$  with any $L$ satisfying \eqref{conditions_setL}.
	\end{enumerate}
	There exist neighbourhoods $W(\bar{v})$, $U(\bar{p})$ such that the Lipschitzian estimate 
	\begin{equation*}
		\| (v_1-v_2)-2[P(v_1,p_1)-P(v_2,p_2)]\|\leq \|v_1-v_2\|+\ell^0\|p_1-p_2\|  
	\end{equation*}
	holds for all $(v_1,p_1),(v_2,p_2)\in W(\bar{v})\times U(\bar{p})$ with some positive constant 
	$\ell^0$.
\end{corollary}

\section{Conclusion}\label{section:conclusion}
In the present paper we proved Lipschitzian stability of projections (in the sense of \eqref{lipschitzian_stability_of_projections})		 onto parametric polyhedral sets in Hilbert space setting with parameters appearing both in left- and right-hand sides of  constraints, which are assumed to be locally Lipschitz. The equality and inequality constraints are allowed. Basic tools for our main results are RCRCQ condition and the representation stability condition (see Definition \ref{def:stability}). 

In general, there is no relationship between RCRCQ and MFCQ (cf. \cite{Kruger2014}).
Moreover, in Propositions \ref{lemma:indices_RCRCQ}, \ref{proposition:noneptiness_lagrange}, Corollary \ref{remark:RCRCQ_extended}, Theorem \ref{theorem:main}  the conclusions depend upon formula \eqref{eq:representation1} and the representation stability condition in which the index set $\bar{K}$ may not be uniquely defined.
\section{Appendix}

\begin{lemma}\label{lemma:linear_independent}(\cite[Lemma 10]{on_lipschitz_like_property})
	Let $J=\{1,\dots,k\}$. Let $g_i:\ \hilbertH\rightarrow \mathcal H$, $i\in J$ be continuous operators and let $\bar{p}$ be such that $g_i(\bar{p})$, $i\in J$ are linearly independent. Then there exists a neighbourhood $U(\bar{p})$ such that for all $p\in U(\bar{p})$, $g_i(p)$, $i\in J$ are linearly independent.
\end{lemma}
\begin{lemma}\label{lemma:positive_combination} (\cite[Lemma 6.33]{deutsch2001best})
	If a nonzero vector $x$ is a positive linear combination of the nonzero vectors $g_1,\dots,g_n$, then $x$ is a positive linear combination of linearly independent subset of $\{g_1,\dots,g_n\}$.
\end{lemma}
\begin{remark}\label{remark:positive_combination}
	Let $J=\{1,\dots,k\}$, $J=W_1\cup W_2$, $W_1\cap W_2=\emptyset$  and let $x=\sum_{i\in J} \lambda_i g_i $, $\lambda_i\leq 0$, $i\in W_1$, $\lambda_i\geq 0$, $i\in W_2$, where $g_i$, $i\in J$ are nonzero vectors. Then there exists $\bar{J}\subset J$ and $\bar{\lambda}_i$, $i\in \bar{J}$ such that
	\begin{equation*}
		x=\sum_{i\in \bar{J}} \lambda_i g_i,\quad \lambda_i< 0,\ i\in \bar{J}\cap W_1, \quad \lambda_i>0,\ i\in \bar{J}\cap W_2
	\end{equation*}
	and $g_i$, $i\in \bar{J}$ are linearly independent.
\end{remark}
\begin{proof}
	We have $x=\sum_{i\in J_1} \lambda_i g_i $, where $J_1\subset J$ and $\lambda_i< 0$, $i\in  J_1\cap W_1$, $\lambda_i>0$, $i\in J_1\cap W_2$. Let
	\begin{equation*}
		\tilde{g}_i=\left\{\begin{array}{rl}
			g_i & \text{if}\ i\in W_2\\
			-g_i & \text{if}\ i\in W_1
		\end{array}\right.\quad i\in J_1,\quad
		\tilde{\lambda}_i=\left\{\begin{array}{rl}
			\lambda_i & \text{if}\ i\in W_2\\
			-\lambda_i & \text{if}\ i\in W_1
		\end{array}\right.\quad i\in J_1.
	\end{equation*}
	Then $x=\sum_{i\in J_1} \tilde{\lambda}_i \tilde{g_i}$ and $\tilde{\lambda}_i$, $i\in J_1$ are positive. Applying Lemma \ref{lemma:positive_combination} we have that there exists $\bar{J}\subset J_1$ and $\hat{\lambda}_i>0$, $i\in \bar{J}$ such that $x=\sum_{i\in \bar{J}} \hat{\lambda}_i \tilde{g_i}$ and $\tilde{g}_i$, $i\in \bar{J}$ are linearly independent. Now let
	\begin{equation*}
		\bar{\lambda}_i=\left\{\begin{array}{rl}
			\hat{\lambda}_i & \text{if}\ i\in W_2\\
			-\hat{\lambda}_i & \text{if}\ i\in W_1
		\end{array}\right.\quad i\in \bar{J}.
	\end{equation*} 
	Then $x=\sum_{i\in \bar{J}} \bar{\lambda}_i g_i $, $\bar{\lambda}_i<0$, $i\in \bar{J}\cap W_1$, $\bar{\lambda}_i>0$, $i\in \bar{J}\cap W_2$.
\end{proof}
\begin{lemma}\label{lemma:positive_combination2}\cite[Lemma 12]{on_lipschitz_like_property}
	Let $x=\sum_{i\in J_1} \lambda_i a_i + \sum_{i\in J_2} \lambda_i a_i$, $J_1\cap J_2=\emptyset$, $J_1,J_2$ finite sets, $\lambda_i\in \mathbb{R}$, $i\in J_1$, $\lambda_i\geq 0$, $i\in J_2$ and $a_i$, $i\in J_1\cup J_2$ are non-zero vectors. Assume that $a_i$, $i\in J_1$ are linearly independent. Then there exists $J_2^\prime\subset J_2$ and $\lambda_i^\prime$, $i\in J_1\cup J_2^\prime$, $\lambda_i^\prime\in \mathbb{R}$, $i\in J_1$, $\lambda_i^\prime>0$, $i\in J_2^\prime$ such that
	\begin{equation*}
		\sum_{i\in J_1} \lambda_i a_i + \sum_{i\in J_2} \lambda_i a_i=\sum_{i\in J_1} \lambda_i^\prime a_i + \sum_{i\in J_2^\prime} \lambda_i^\prime a_i
	\end{equation*} 
	and $a_i$, $i\in J_1\cup J_2^\prime$ are linearly independent.
\end{lemma}

\begin{lemma}\label{lemma:gramschmidt} Let $\hilbertH$ be a Hilbert space.
	Let $\{u_i^n\}_{n\in\mathbb{N}}$, $i\in 1,\dots,K$ be a sequence of $K$-tuples of vectors from $\hilbertH$ such that for any $n\in\mathbb{N}$, $u_i^n$, $i\in 1,\dots,K$ are linearly independent. 
	Assume that 
	\begin{enumerate}[(1)]
		\item $u_i^n\rightarrow u_i$, for $i=1\dots,K$, where $u_i$, $i=1\dots,K$, are linearly independent,
		\item $\sum_{i=1}^{K} \lambda_i^n u_i^n\rightarrow \sum_{i=1}^{K} \bar{\lambda}_i {u}_i$, where $\lambda_i^n,\bar{\lambda}_i\in \mathbb{R}$, $i=1,\dots K$, $n\in\mathbb{N}$.
	\end{enumerate}
	Then$\lambda_i^n\rightarrow \bar{\lambda}_i$, $i=1,\dots,K$.

\end{lemma}

\begin{proof} 
	For any $n\in \mathbb{N}$ let $\{e_i^n\}_{n\in\mathbb{N}}$ $i=1,\dots,K$ be a sequence of orthogonal vectors obtained by the Gram-Schmidt orthogonalization of $u_i^n$ $i=1,\dots,K$, i.e.,
	\begin{align*}
		&e_1^n=u_1^n,\\
		&e_k^n=u_k^n-\sum_{i=1}^{k-1} \frac{\langle u_k^n \ |\ e_i^n \rangle }{\|e_i^n\|^2}e_i^n,\quad k=2,\dots,K
	\end{align*}
	and $e_i$, $i=1,\dots,K$, be orthogonal vectors obtained  by the Gram-Schmidt orthogonalization of $u_i$ $i=1,\dots,K$.
	Since $u_i^n\rightarrow u_i$, $i=1,\dots,K$ we have $e_i^n\rightarrow e_i$, $i=1,\dots,K$.
	
	Let $\tilde{\lambda}_i^n,\tilde{\lambda}_i\in \mathbb{R}$, $i=1,\dots K$, $n\in\mathbb{N}$  be such that $\sum_{i=1}^{K} \lambda_i^n u_i^n=\sum_{i=1}^{K} \tilde{\lambda}_i^n e_i^n$ and $\sum_{i=1}^{K} \bar{\lambda}_i u_i=\sum_{i=1}^{K} \tilde{\lambda}_i e_i$. 
	Since 
	$
	\sum_{i=1}^K ( \lambda_i^n u_i^n-\bar{\lambda}_i u_i)\rightarrow 0
	$
	we have
	\begin{equation}\label{eq:lambdabounded1}
		\sum_{i=1}^K ( \tilde{\lambda}_i^n e_i^n-\tilde{\lambda}_i e_i)\rightarrow 0.
	\end{equation}
	We have
	\begin{equation}\label{eq:lambdabounded2}
		\sum_{i=1}^K ( \tilde{\lambda}_i^n e_i^n-\tilde{\lambda}_i e_i)=
		\sum_{i=1}^K  \tilde{\lambda}_i^n (e_i^n-e_i) + \sum_{i=1}^K (\tilde{\lambda}_i^n  -\tilde{\lambda}_i) e_i
	\end{equation}
	By taking  scalar products with $e_{k}$, $k\in\{1,\dots,K\}$ we obtain
	\begin{equation*}
		\sum_{i=1}^K
		\tilde{\lambda}_i^n \langle e_i^n-e_i \ |\ e_{k} \rangle + (\tilde{\lambda}_k^n  -\tilde{\lambda}_k)\|e_{k}\|^2\rightarrow 0, \quad k=1,\dots,K.
	\end{equation*}
	Now we show that $\tilde{\lambda}_i^n$ are bounded, i.e. $|\tilde{\lambda}_i^n|\leq M$ for $i=1,\dots,K$, $n\in \mathbb{N}$ for some $M\geq 0$. We have
	\begin{equation*}
		\|\sum_{i=1}^K\tilde{\lambda}_i^n e_i^n\|^2= 	\sum_{i=1}^K (\tilde{\lambda}_i^n )^2 \|e_i^n\|^2\geq (\tilde{\lambda}_k^n)^2 \|e_k^n\|^2\quad \text{for all } k\in\{1,\dots,K\}
	\end{equation*}
	If there exists $k \in\{1,\dots,K\}$ such that $|\tilde{\lambda}_n^k|\rightarrow +\infty$, then $\|\sum_{i=1}^K\tilde{\lambda}_i^n e_i^n\|^2\rightarrow +\infty$, which contradicts $\|\sum_{i=1}^K\tilde{\lambda}_i^n e_i^n\|^2\rightarrow\|\sum_{i=1}^K\tilde{\lambda}_i e_i\|^2= \|\sum_{i=1}^K\lambda_i u_i\|^2<+\infty$.

	{Since $\tilde{\lambda}_i^n$ are bounded} and $e_i^n-e_i\rightarrow 0$, $i=1,\dots,K$ from \eqref{eq:lambdabounded1}-\eqref{eq:lambdabounded2} we conclude that $\tilde{\lambda}_k^n\rightarrow \tilde{\lambda}_k$ for any $k\in \{1,2\dots,K\}$.
	
	For each $n\in \mathbb{N}$ we have
	\begin{equation*}
		\sum_{i=1}^{K} \lambda_i^n u_i^n=	\sum_{i=1}^{K} \tilde{\lambda}_i^n e_i^n=\tilde{\lambda}_1^n u_1^n+\sum_{i=2}^{K} \tilde{\lambda}_i^n \left( 
		u_i^n-\sum_{j=1}^{i-1} \frac{\langle u_i^n \ |\ e_j^n \rangle }{\|e_j^n\|^2}e_j^n\right)
	\end{equation*}
	Since $u_i^n$, $i=1,\dots,K$ are linearly independent we have
	\begin{align*}
		&\lambda_K^n u_K^n=\tilde{\lambda}_K^n u_K^n,\\
		&\lambda_{K-1}^n u_{K-1}^n=\tilde{\lambda}_{K-1}^n u_{K-1}^n-\tilde{\lambda}_{K}^n\frac{\langle u_{K}^n \ |\ e_{K-1}^n \rangle }{\|e_{K-1}^n\|^2}u_{K-1}^n,\\
		&\lambda_{K-2}^n u_{K-2}^n=\tilde{\lambda}_{K-2}^n u_{K-2}^n-\tilde{\lambda}_{K-1}^n \frac{\langle u_{K-1}^n \ |\ e_{K-2}^n \rangle }{\|e_{K-2}^n\|^2}u_{K-2}^n\\
		&\hspace{1cm}
		+\tilde{\lambda}_{K}^n\bigg( \frac{\langle u_{K}^n \ | \ e_{K-1}^n \rangle  }{\|e_{K-1}^n\|^2} \frac{\langle u_{K-1}^n \ |\ e_{K-2}^n \rangle  }{{\|e_{K-2}^n\|^2}  } u_{K-2}^n \\
		&\hspace{1cm} - \frac{\langle u_K^n \ |\ e_{K-2}^n \rangle  }{\|e_{K-2}^n\|^2}u_{K-2}^n  \bigg)\\
		&\hspace{5cm}\vdots
	\end{align*}
	Hence,
	\begin{align*}
		&\lambda_K^n=\tilde{\lambda}_K^n,\\
		&\lambda_{K-1}^n=\tilde{\lambda}_{K-1}^n-\tilde{\lambda}_{K}^n\frac{\langle u_{K}^n \ |\ e_{K-1}^n  \rangle}{\|e_{K-1}^n\|^2},\\
		&\lambda_{K-2}^n=\tilde{\lambda}_{K-2}^n- \tilde{\lambda}_{K-1}^n \frac{\langle u_{K-1}^n \ |\ e_{K-2}^n \rangle }{\|e_{K-2}^n\|^2} \\
		&\hspace{1cm} + \tilde{\lambda}_{K}^n\left( \frac{\langle u_{K}^n \ | \ e_{K-1}^n \rangle  }{\|e_{K-1}^n\|^2} \frac{\langle u_{K-1}^n \ |\ e_{K-2}^n \rangle  }{{\|e_{K-2}^n\|^2}  } - \frac{\langle u_K^n \ |\ e_{K-2}^n \rangle  }{\|e_{K-2}^n\|^2}\right)\\
		&\hspace{5cm}\vdots
	\end{align*}
	Since $e_i^n\rightarrow e_i$ , $u_i^n\rightarrow u_i$ and $\tilde{\lambda}_i^n\rightarrow\tilde{\lambda}_i$ for any $i\in \{1,\dots,K\}$  we obtain that $\lambda_i^n $ converges for any $i\in \{1,\dots,K\}$ and $\lambda_i^n\rightarrow \lambda_i$, $i=1,\dots,K$.
\end{proof}
\bibliographystyle{plain}
\bibliography{mybibfile}
\end{document}